\documentclass{amsart}
\newcommand{\RN}[1]{%
  \textup{\uppercase\expandafter{\romannumeral#1}}%
}

\usepackage{amsmath}
\usepackage{amssymb}
\usepackage{enumerate}
\usepackage[hidelinks]{hyperref}
\usepackage{mathtools}
\usepackage[capitalise]{cleveref}
\numberwithin{equation}{section}

\crefname{equation}{}{}
\newtheorem{theorem}{Theorem}[section]
\newtheorem{lemma}[theorem]{Lemma}
\newtheorem{proposition}[theorem]{Proposition}
\newtheorem{corollary}[theorem]{Corollary}
\newtheorem{conjecture}[theorem]{Conjecture}

\theoremstyle{definition}
\newtheorem{definition}[theorem]{Definition}
\newtheorem{example}[theorem]{Example}

\theoremstyle{remark}
\newtheorem{remark}[theorem]{Remark}

\usepackage{enumerate}
\usepackage{tikz}
\usepackage{tikz-cd}
\usetikzlibrary{shapes,arrows}
\tikzstyle{line}=[draw] 
\usepackage{enumitem}
\usepackage{epigraph}
\usepackage[utf8]{inputenc}
\usepackage{newunicodechar}
\newunicodechar{Δ}{\Delta}
\newunicodechar{×}{\times}
\newunicodechar{τ}{\tau}
\newunicodechar{⊗}{\otimes}
\newunicodechar{⊕}{\oplus}
\newunicodechar{∘}{\circ}
\newunicodechar{•}{\bullet}
\newunicodechar{→}{\to}
\newunicodechar{ℳ}{\mathcal{M}}
\newunicodechar{𝒪}{\mathcal{O}}
\newunicodechar{ℱ}{\mathcal{F}}
\newunicodechar{⋯}{\cdots}
\newunicodechar{∧}{\wedge}
\newunicodechar{∨}{\vee}
\newunicodechar{≅}{\cong}
\newunicodechar{𝖅}{\mathfrak{Z}}
\newunicodechar{Δ}{\Delta}
\newunicodechar{ℙ}{\mathbb{P}}
\newunicodechar{𝒰}{\mathcal{U}}
\newunicodechar{𝓛}{\mathcal{L}}
\newunicodechar{𝖄}{\mathfrak{Y}}
\newunicodechar{⊂}{\subset}
\newunicodechar{α}{\alpha}
\newunicodechar{𝓥}{\mathcal{V}}
\newunicodechar{𝓦}{\mathcal{W}}
\newunicodechar{ℤ}{\mathbb{Z}}
\newunicodechar{ℕ}{\mathbb{N}}
\newunicodechar{𝓘}{\mathcal{I}}
\author{Yu Zhao}
\title{Derived Blow-ups and Birational Geometry of Nested Quiver Varieties}
\date{\today}

\newcommand\mf{\mathfrak}
\newcommand\mc{\mathcal}
\newcommand\mb{\mathbb}
\newcommand\mbf{\mathbf}
\begin{document}
\maketitle
\begin{center}
    \footnotesize\it  Dedicated to Professor Hiraku Nakajima on the occasion of his 60th birthday
\end{center}

\begin{abstract}
  Given a quiver, Nakajima introduced the quiver variety and the Hecke correspondence, which is a closed subvariety of Cartesian products of quiver varieties. In this paper, we consider two nested quiver varieties as fiber products of Hecke correspondences along natural projections. After blowing up the diagonal, we prove that they are isomorphic to a quadruple moduli space which Negu\c{t} observed for the Jordan quiver. The main ingredient is Jiang's derived projectivization theory and Hekking's derived blow-up theory, while we obtain the local model when both the two derived schemes are quasi-smooth. 
\end{abstract}
\section{Introduction}

\subsection{The main theorem}

Let $Q=(I,E)$ be a quiver and $\mbf{w}\in \mb{Z}_{\geq 0}^{I}$ be a vector with $\mbf{w}\neq 0$. In a series of great work \cite{10.1215/S0012-7094-94-07613-8}\cite{10.1215/S0012-7094-98-09120-7}\cite{10.2307/2646205}, Nakajima
\begin{enumerate}
\item introduced the Nakajima quiver variety $\mf{M}(\mbf{v},\mbf{w})$, given $\mbf{v}\in \mb{Z}_{\geq  0}^{I}$;
\item introduced the Hecke correspondence $\mf{P}(\mbf{v}^{1},\mbf{v}^{2})$ (see \cref{def:4.1} for the definition), which is a closed smooth subvariety of $\mf{M}(\mbf{v}^{1},\mbf{w})\times \mf{M}(\mbf{v}^{2},\mbf{w})$. Here $\mbf{v}^{1},\mbf{v}^{2}\in \mb{Z}_{\geq 0}^{I}$ and there exists $k\in I$ such that $\mbf{v}^{2}-\mbf{v}^{1}=(\delta_{kj})_{j\in I}$, where $\delta_{kj}$ is the Kronecker symbol;
\item  proved that the tautological line bundles on $\mf{P}(\mbf{v}^{1},\mbf{v}^{2})$ generate an action of quantum loop groups on the Grothendieck group of quiver varieties if $Q$ has no edge loops (the similar construction for cohomology was also obtained by Nakajima \cite{10.1215/S0012-7094-98-09120-7}).
\end{enumerate}

Given $k\in I$ and $\mbf{v}^{0},\mbf{v}^{1},\mbf{v}^{1'},\mbf{v}^{2}\in \mb{Z}_{\geq 0}^{I}$ such that $\mbf{v}^{1}=\mbf{v}^{1'}$ and  $\mbf{v}^{2}-\mbf{v}^{1}=\mbf{v}^{1'}-\mbf{v}^{0}=(\delta_{kj})_{j\in I}$, in this paper we study the geometry of the following triple moduli spaces
\begin{align*}
  \mf{Z}_{k}^{-}(\mbf{v}^{1}):=\mf{P}(\mbf{v}^{0},\mbf{v}^{1})\times_{\mf{M}(\mbf{v}^{0},\mbf{w})}\mf{P}(\mbf{v}^{0},\mbf{v}^{1'})  \\
 \mf{Z}_{k}^{+}(\mbf{v}^{1}):=\mf{P}(\mbf{v}^{1},\mbf{v}^{2})\times_{\mf{M}(\mbf{v}^{2},\mbf{w})}\mf{P}(\mbf{v}^{1'},\mbf{v}^{2}).
\end{align*}

Let
$$\Delta_{p(\mbf{v}^{0},\mbf{v}^{1})}:\mf{P}(\mbf{v}^{0},\mbf{v}^{1})\to \mf{Z}_{k}^{-}(\mbf{v}^{1}),\quad \Delta_{q(\mbf{v}^{1},\mbf{v}^{2})}:\mf{P}(\mbf{v}^{1},\mbf{v}^{2})\to \mf{Z}_{k}^{+}(\mbf{v}^{1}).$$
be the diagonal embeddings of the fiber products respectively. It is well-known that (see \cref{012531} or \cite{MR1711344}) that $\mf{Z}^{-}_{k}(\mbf{v})-\mf{P}(\mbf{v}^{0},\mbf{v}^{1})\cong \mf{Z}^{+}_{k}(\mbf{v})-\mf{P}(\mbf{v}^{1},\mbf{v}^{2})$.

Given a closed embedding of schemes $f:X\to Y$, we denote $Bl_{f}$ as the blow-up of $X$ in $Y$. In this paper, we prove that

\begin{theorem}[\cref{0213711}]
  \label{thm:main}
  There exists a smooth variety $\mf{Y}_{k}(\mbf{v}^{1})$ such that
  \begin{equation}
    \label{eq:1.1}
  Bl_{\Delta_{p(\mbf{v}^{0},\mbf{v}^{1})}}\cong \mf{Y}_{k}(\mbf{v}^{1})\cong Bl_{\Delta_{q(\mbf{v}^{1},\mbf{v}^{2})}},    
  \end{equation}
and \cref{eq:1.1} has an enhancement in derived schemes. Moreover, if $dim(\mf{P}(\mbf{v}^{0},\mbf{v}^{1}))>dim(\mf{M}(\mbf{v}^{0},\mbf{w}))$ (resp. $dim(\mf{P}(\mbf{v}^{1},\mbf{v}^{2}))>dim(\mf{M}(\mbf{v}^{2},\mbf{w}))$), then $\mf{Z}_{k}^{-}(\mbf{v}^{1})$ (resp. $\mf{Z}_{k}^{+}(\mbf{v}^{1})$) is a canonical singularity.
\end{theorem}

In \cref{intro2}, \cref{intro3} and \cref{intro4}, we will introduce the geometry of $\mf{Y}_{k}(\mbf{v}^{1})$ and the meaning of the derived enhancement.

\subsection{Negu\c{t}'s quadruple moduli space $\mf{Y}_{k}(\mbf{v}^{1})$}
\label{intro2}
The smooth variety $\mf{Y}_{k}(\mbf{v}^{1})$ can be realized as a closed subvariety of
$$\mf{M}(\mbf{v}^{0},\mbf{w})\times \mf{M}(\mbf{v}^{1},\mbf{w})\times \mf{M}(\mbf{v}^{1'},\mbf{w})\times \mf{M}(\mbf{v}^{2},\mbf{w})$$
in the following way: given $k\in I$, let $g_{k}$ be the number of edge loops with source $k$. For any ADHM datum $[B,i,j]$ (see \cref{s020915} for the definition of ADHM datum) and $l$ is an edge loop in $E^{\#}:=E\cup \bar{E}$ with source $k$. The trace $tr(B_{l})$ is invariant under the gauge group action. Hence we could consider the function $tr_{k}$ such that
$$tr_{k}([B,i,j])=(tr(B_{l}))_{l\in E^{\#},in(l)=out(l)=k}$$

\begin{definition}[\cref{021354}, Negu\c{t} \cite{negut2017shuffle}\cite{neguct2018hecke} for the Jordan quiver]
  \label{def:1.2}
  We define $\mf{Y}_{k}(\mbf{v}^{1})$ as the variety which consists of quadruples
  $$(s^{0},s^{1},s^{1'},s^{2})\in \mf{M}(\mbf{v}^{0},\mbf{w})\times \mf{M}(\mbf{v}^{1},\mbf{w})\times \mf{M}(\mbf{v}^{1'},\mbf{w})\times \mf{M}(\mbf{v}^{2},\mbf{w})$$
  such that
  \begin{enumerate}
  \item $(s^{0},s^{1})\in \mf{P}(\mbf{v}^{0},\mbf{v}^{1})$ and $(s^{0},s^{1'})\in \mf{P}(\mbf{v}^{0},\mbf{v}^{1'})$
  \item $(s^{1},s^{2})\in \mf{P}(\mbf{v}^{1},\mbf{v}^{2})$ and $(s^{1'},s^{2})\in \mf{P}(\mbf{v}^{1'},\mbf{v}^{2})$
  \item $tr_{k}(s^{1})-tr_{k}(s^{0})=tr_{k}(s^{2})-tr_{k}(s^{1'})$.
  \end{enumerate}
\end{definition}
\begin{remark}
    \cref{def:1.2} was first observed by Negu\c{t} \cite{negut2017shuffle}\cite{neguct2018hecke} for the Jordan quiver. We reformulate his definition to make it well-defined for all the quiver varieties and prove its smoothness in \cref{021147}.
\end{remark}

By forgetting $s^{2}$ and $s^{0}$ respectively, we induce morphisms
$$\alpha^{-}:\mf{Y}_{k}(\mbf{v}^{1})\to \mf{Z}_{k}^{-}(\mbf{v}^{1}),\quad \alpha^{+}:\mf{Y}_{k}(\mbf{v}^{1})\to \mf{Z}_{k}^{+}(\mbf{v}^{1})$$
In \cref{0213711}, we strengthen \cref{thm:main} by proving that $\alpha^{\pm}$ are projection morphisms of respective blow-ups.

\begin{remark}
  Nakajima \cite{10.1215/S0012-7094-98-09120-7} proved that if $Q$ has no edge loops, $\mf{P}(\mbf{v}^{1},\mbf{v}^{2})$ is a Lagragian of $\mf{M}(\mbf{v}^{1},\mbf{w})\times \mf{M}(\mbf{v}^{2},\mbf{w})$. When $Q$ has edge loops, we need a slight modification: we consider the function
  $$t(\mbf{v}^{1},\mbf{v}^{2}):\mf{P}(\mbf{v}^{1},\mbf{v}^{2})\to \mb{C}^{2g_{k}},\quad (s^{1},s^{2})\to tr_{k}(s^{2})-tr_{k}(s^{1}).$$
In \cref{020527}, we prove that under the closed embedding $p\times q\times t(\mbf{v}^{1},\mbf{v}^{2})$, $\mf{P}(\mbf{v}^{1},\mbf{v}^{2})$ is a Lagragian of $\mf{M}(\mbf{v}^{1},\mbf{w})\times \mf{M}(\mbf{v}^{2},\mbf{w})\times \mb{C}^{2g_{k}}$.
\end{remark}
\subsection{Jiang's derived projectivization and Hekking's derived blow-up theory}
\label{intro3}
Let $f:X\to Y$ be a closed embedding of smooth varieties. Let $C_{f}$ be the co-normal bundle of $X$ in $Y$. Then we have the following commutative diagram:
\begin{equation}
  \label{eq:1.2}
  \begin{tikzcd}
    P_{X}(C_{f})\ar{r}\ar{d} & Bl_{f}\ar{d} \\
    X \ar{r} & Y
  \end{tikzcd}
\end{equation}
such that $P_{X}(C_{f})$ is the exceptional divisor of $Bl_{f}$ and $Bl_{f}-P_{X}(C_{f})\cong Y-X$.

We want to extend the diagram \cref{eq:1.2} to a closed embedding of quasi-smooth (derived) schemes, which require a universal theory of the projectivization of $F\in QCoh(X)$ when $X$ is a derived scheme and a universal theory of the derived blow-up for a closed embedding of derived schemes. They were realized recently by Jiang \cite{jiang2022derived} and Hekking \cite{Hekking_2022}\cite{hekking2021graded} respectively. In this paper we denote
 $$\mb{P}_{X}(F)$$
 where $F\in QCoh(X)$ as the derived projectivization (or Jiang's projectivization), and denote
 $$\mb{B}l_{f},$$
 where $f$ is a closed embedding of derived schemes, as the derived blow-up (or Hekking's blow-up), to distinguish with Grothendieck's projectivization or the classical blow-up theory. 
\begin{theorem}[Theorem 5.3 and Proposition 5.5 of \cite{Hekking_2022}]
Let $L_{f}$ be the cotangent complex of $f$, and we denote the conormal complex $C_{f}:=L_{f}[-1]$. Then we have a commutative diagram
\begin{equation*}
   \begin{tikzcd}
    \mb{P}_{X}(C_{f})\ar{r}\ar{d} & \mb{B}l_{f}\ar{d}{pr_{f}} \\
    X \ar{r} & Y
  \end{tikzcd}
\end{equation*}
such that $\mb{P}_{X}(C_{f})$ is a (virtual) Cartier divisor of $\mb{B}l_{f}$ and $\mb{B}l_{f}-\mb{P}_{X}(C_{f})\cong Y-X$. The projection morphism $pr_{f}$ is proper.
\end{theorem}
\begin{remark}
  Before Jiang and Hekking, there were already many attempts to describe the derived projectivization or blow-up in classical geometry (under certain circumstances), like Kirwan \cite{10.2307/1971369}, Kiem-Li \cite{10.2307/43302830}, Vakil-Zinger \cite{10.2140/gt.2008.12.1}, Savvas \cite{savvas2018generalized}, Negu\c{t} \cite{negut2017shuffle}, Jiang-Leung \cite{jiang2018derived}, Jiang \cite{jiang_2021}, Addington-Takahashi \cite{addington2020categorical} and the author \cite{zhao2020categorical}\cite{zhao2021moduli}.
\end{remark}

\subsection{Derived blow-up for quasi-smooth derived schemes}
\label{intro4}
 
We propose the following conjecture about the Tor-amplitude:
\begin{conjecture}
    Given an integer $n\geq 0$ and a closed embedding of derived schemes $f:X\to Y$. If $L_{X}$ and $L_{Y}$ both have Tor-amplitude $[0,n]$, then $L_{\mb{B}l_{f}}$ also has Tor-amplitude $[0,n]$.
\end{conjecture}
In the case that $n=0$, the proof is standard, like Theorem 22.3.10 of \cite{vakil2017rising}. We prove the case that $n=1$:
\begin{theorem}[\cref{cor11}]
  If both $X$ and $Y$ are quasi-smooth, the derived blow-up $\mb{B}l_{f}$ is also quasi-smooth.
\end{theorem}

From now on we assume that both $X$ and $Y$ are quasi-smooth. For a quasi-smooth derived scheme $M$,  Ciocan-Fontanine and Kapranov \cite{ciocan-fontanine07:virtual} defined $K$-theoretic virtual fundamental class $[M]_{K}^{vir}$ on the underlying scheme $\pi_{0}(M)$. As $Y$ are $\mb{B}l_{f}$ are both quasi-smooth, it is natrual to compare $[\mb{B}l_{f}]_{K}^{vir}$ and $[Y]_{K}^{vir}$ through the push-forward of $pr_{f}$.

Let $r:=vdim(X)-vdim(Y)$ be the relative virtual dimension. When $r<0$, we prove that
\begin{theorem}[\cref{020655},  vanishing theorem]
  \label{thm:van}
  Given $r<l\leq  0$, we have
  $$\mbf{R}pr_{f*}\mc{O}_{\mb{B}l_{f}}(l)\cong \mc{O}_{Y}.$$
Here $\mbf{R}pr_{f*}$ is the derived push-forward functor.
\end{theorem}
As a corollary, we have
$$ \pi_{0}(pr_{f})_{*}[\mb{B}l_{f}]_{K}^{vir}=[Y]_{K}^{vir}.$$

If $X$ is smooth and the relative dimension $r:=vdim(X)-vdim(Y)\geq 0$, we propose the following conjecture

\begin{conjecture}
    \label{conj:1}
  There exists $D_{l}\in D^{b}_{coh}(Y)$, $-r-1\leq l\leq  0$ such that
  \begin{enumerate}
  \item $D_{-r-1}\cong\mc{O}_{Y}$ and $D_{0}\cong \mbf{R}pr_{f*}\mc{O}_{\mb{B}l_{f}}$;
  \item We have triangles:
    \begin{equation}
      \label{eq:triangle}
     D_{l-1}\to D_{l}\to \mbf{R}f_{*}(det(C_{f})^{-1}\wedge^{l-r}(C_{f}^{\vee}[1]))[1-l]
    \end{equation}
  \end{enumerate}
\end{conjecture}
If \cref{conj:1} is true, we have the following enumerative counting formula:
\begin{equation}
  \label{eq:noref}
  \pi_{0}(pr_{f})_{*}[\mb{B}l_{f}]_{K}^{vir}=[Y]_{K}^{vir}+f_{*}(\sum_{l=-r}^{0}(-1)^{l-1}det(C_{f})^{-1}\wedge^{l+r}(C_{f}^{\vee}[1])).
\end{equation}
In \cref{sec:app}, we prove \cref{conj:1} if $Y$ is the derived zero locus of $f\in \Gamma(Z,V)$ where $V$ is a locally free sheaf over a smooth scheme $Z$.

Now we describe the birational geometry of $\mb{B}l_{f}$ and the comparison with $Bl_{\pi_{0}(f)}$. We have that
\begin{theorem}[Theorem 3.5.5 of \cite{Hekking_2022} and \cref{020651}]
  The scheme $Bl_{\pi_{0}(f)}$ is a closed subscheme of $\pi_{0}(\mb{B}l_{f})$. If $\pi_{0}(Y)-\pi_{0}(X)$ is not empty and $\mb{B}l_{f}$ is smooth, then $Bl_{\pi_{0}(f)}\cong \mb{B}l_{f}$.
\end{theorem}

In this paper, we give a criterion for the smoothness of $\mb{B}l_{f}$, 
\begin{theorem}[\cref{020656}, Gluing theorem]
  If both $\mb{P}_{Y}(C_{f})$ and $Y-X$ are smooth, then the derived blow up $\mb{B}l_{f}$ is also smooth. 
\end{theorem}

As $Y$ and $\mb{B}l_{f}$ are both quasi-smooth, by Sch\"{u}rg-To\"{e}n-Vezzosi \cite{SchurgToenVezzosi+2015+1+40} and Jiang \cite{jiang2022derived}, the determinant of $L_{X}$ and $L_{\mb{B}l_{f}}$ are well defined, which we denote as $K_{Y}$ and $K_{\mb{B}l_{f}}$ respectively. If $\mb{B}l_{f}$ is smooth and $\pi_{0}(Y)-\pi_{0}(X)$ is not empty, we have 
\begin{theorem}[Discrepancy formula   \cref{lem:1.5}]
  \label{disc}
  We have $K_{\mb{B}l_{f}}\cong \mbf{L}pr_{f}^{*}K_{Y}\otimes \mc{O}_{\mb{B}l_{f}}(vdim(X)-vdim(Y)+1)$.
\end{theorem}
\begin{corollary}[Grauert-Riemenschneider vanishing theorem, \cref{grau}]
  \label{cor:gra}
  If the relative virtual dimension $vdim(X)-vdim(Y)<0$, $\mb{B}l_{f}$ is smooth and $\pi_{0}(Y)-\pi_{0}(X)$ is not empty, we have $\mbf{R}pr_{f*}K_{\mb{B}l_{f}}\cong K_{Y}$.
\end{corollary}
\begin{conjecture}
  \label{conj:2}
We can remove the assumptions that $\mb{B}l_{f}$ is smooth and $\pi_{0}(Y)-\pi_{0}(X)$ is not empty in \cref{disc} and \cref{cor:gra}.  
\end{conjecture}
\begin{remark}
  To prove the theorems in this subsection, we develop the local theory of derived blow-ups, i.e. describe the (Zariski) local neighborhood of $\mb{B}l_{f}$. To prove \cref{conj:1} and \cref{conj:2}, we will need a descent theory (or a global theory), to prove that the local computations descend to a global one.
\end{remark}

Finally, we apply the derived blow-up theory to the triple moduli spaces.
We consider the following derived enhancement of $\mf{Z}_{k}^{\pm}(\mbf{v}^{1})$
$$\mb{R}\mf{Z}_{k}^{-}(\mbf{v}^{1}):=\mf{P}(\mbf{v}^{1},\mbf{v}^{2})\times^{\mb{L}}_{\mf{M}(\mbf{v}^{0},\mbf{w})}\mf{P}(\mbf{v}^{1'},\mbf{v}^{2}),\quad \mb{R}\mf{Z}_{k}^{+}(\mbf{v}^{1}):=\mf{P}(\mbf{v}^{1},\mbf{v}^{2})\times_{\mf{M}(\mbf{v}^{2},\mbf{w})}^{\mb{L}}\mf{P}(\mbf{v}^{1'},\mbf{v}^{2})$$
together with the diagonal embeddings
\begin{align*}
  \mb{R}\Delta_{p(\mbf{v}^{0},\mbf{v}^{1})}:\mf{P}(\mbf{v}^{0},\mbf{v}^{1})\to \mb{R}\mf{Z}_{k}^{-}(\mbf{v}^{1}),\quad \mb{R}\Delta_{q(\mbf{v}^{1},\mbf{v}^{2})}:\mf{P}(\mbf{v}^{1},\mbf{v}^{2})\to \mb{R}\mf{Z}_{k}^{+}(\mbf{v}^{1}).
\end{align*}
The following theorem gives the enhancement of the blow-up of the diagonal embeddings:
\begin{theorem}[\cref{020658}]
  Both $\mb{B}l_{\mb{R}\Delta_{p(\mbf{v}^{0},\mbf{v}^{1})}}$ and $\mb{B}l_{\mb{R}\Delta_{q(\mbf{v}^{1},\mbf{v}^{2})}}$ are smooth schemes, and we have
$$\mb{B}l_{\mb{R}\Delta_{p(\mbf{v}^{0},\mbf{v}^{1})}}\cong Bl_{\Delta_{p(\mbf{v}^{0},\mbf{v}^{1})}}, \quad \mb{B}l_{\mb{R}\Delta_{q(\mbf{v}^{1},\mbf{v}^{2})}}\cong Bl_{\Delta_{q(\mbf{v}^{1},\mbf{v}^{2})}}.$$
\end{theorem}

\subsection{Relation with the quantum loop groups}
After Nakajima \cite{10.2307/2646205}, there are several approaches generalizing the quantum loops to quivers with edge loops, like Maulik-Okounkov \cite{MR3951025}, Schiffmann-Vasserot \cite{schiffmann17:hall} and Kontsevich-Soibelman \cite{kontsevich2010cohomological}. Given a quiver $Q=(I,E)$ with a maximal set of deformation parameters, Negu\c{t}-Sala-Schiffmann \cite{neguct2021shuffle} defined a quantum loop group $U_{Q}^{+}$, with generators indexed by $I\times \mb{Z}$ and some explicit quadratic and cubic relations. Moreover, they proved that $U_{Q}^{+}$ is isomorphic to localized $K$-theoretic Hall algebra of $Q$, and thus acts on the Grothendieck group of Nakajima quiver varieties. Moreover, they constructed a Drinfeld double of $U_{Q}^{+}$, which they denoted as $U$.

As a corollary of \cref{thm:main}, in \cite{zhao2023} we will prove that
\begin{theorem}[\cite{zhao2023}]
  The action of $U_{Q}^{+}$ on the Grothendieck group of quiver varieties can be generalized to an action of $U$ through the Hecke correspondences. Moreover, the (categorical) commutator of positive and negative parts can be explicitly represented by complexes of tautological locally free sheaves on the quiver varieties.
\end{theorem}

\subsection{Relation with the stable objects in an abelian category of projective dim $2$}
While this paper focuses on quiver varieties, i.e. the moduli space  of stable representations of preprojective algebras, all the machinery in this paper could be easily generalized to the stable objects in an abelian category of projective dimension $2$, under mild conditions of Chern characters. We will explain the details in further work, and see \cite{zhao2020categorical}\cite{zhao2021moduli} for the abelian category of coherent sheaves over an algebraic surface. 

\subsection{Organization} In \cref{sec2}, we introduce the cohomology theory of preprojective algebra representations. In \cref{sec3}, we introduce Crawley-Boevey's construction and Nakajima quiver variety. In \cref{sec4}, we introduce the Hecke correspondence. In \cref{sec5}, we introduce several nested quiver varieties. We also introduce Negu\c{t}'s quadruple moduli space and prove its smoothness. In \cref{sec6}, we introduce the derived blow-up for the diagonal and prove \cref{thm:main}. In the appendix, we review Jiang's projectivization theory and Hekking's blow-up theory, with a focus on quasi-smooth schemes. We also explain the relationship between the derived blow-up and the virtual fundamental class.

\subsection{Notations}

\subsubsection{Quivers}
A quiver $Q=(I,E)$ is a directed graph, where $I$ is the vertex set and $E$ is the arrow set. We denote $in,out: E\to I$ as the source and target map. For any $e\in E$, we define $\bar{e}$ as an arrow with opposite orientation, i.e. $in(e)=out(\bar{e})$ and $in(\bar{e})=out(e)$. We denote $\overline{E}=\{\bar{e}|e\in E\}$ and $E^{\#}=E\cup \overline{E}$. We denote $Q^{\#}=(I,E^{\#})$ as the associated double quiver. For $\bar{e}\in \overline{E}$, we denote $\overline{\bar{e}}:=e$.

An edge $e\in E^{\#}$ is called a loop if $in(e)=out(e)$, and we say that the origin of $e$ is $in(e)$. Given $k\in I$, we denote the quiver $Q_{k}:=(\{k\},E_{k})$, where
$$E_{k}:=\{e\in E|in(e)=out(e)=k\}.$$
We denote $g_{k}$ as the number of loops in $E_{k}$.
\subsubsection{$I$-graded vector spaces and dimension vectors}
Let $I$ be a finite set. We denote a dimension vector $\mbf{v}$ as an element in $\mb{Z}_{\geq 0}^{I}$. We denote an $I$-graded vector space as $H=(H_{k})_{k\in I}$, where all $H_{k}$ are finite dimensional vector spaces. We define $dim(H)=(dim(H_{k}))_{k\in I}$ as the dimension of $H$. We say $H'$ is an $I$-graded subspace of $H$ if $H'=(H_{k}')_{k\in I}$, where for each $k\in I$, $H_{k'}$ is a subspace oh $H_{k}$.

Given two $I$-graded vector spaces $G,H$, we denote the $I$-graded vector space $I(G,H)=\oplus_{k}Hom(G_{k},H_{k})$. An element $\zeta\in I(G,H)$ is called an $I$-graded homomorphism. We denote $ker(\zeta):=\oplus_{k\in I} ker(\zeta_{k})$ which is an $I$-graded subspace of $G$. We denote $im(\zeta):=\oplus_{k\in I}im(\zeta_{k})$, which is an $I$-graded subspace of $H$. We denote $coker(\zeta):=\oplus_{k\in I}coker(\zeta_{k})$, which is also $I$-graded.

  Given an $I$-graded vector space $G$ and $C\in I(G,G)$, we denote $tr_{I}(C):=(tr(C_{h}))_{h\in I}\in \mb{C}^{I}$ and $tr(C):=\sum_{h\in I}tr(C_{h})\in \mb{C}$.

For $k\in I$, we define the dimension vector $\delta_{k}:=(\delta_{ki})_{i\in I}$, where $\delta_{kj}$ is the Kronecker symbol. For each dimension vector $\mbf{v}$, we can associate an $I$-graded vector space $\mb{C}^{\mbf{v}}:=(\mb{C}^{\mbf{v}_{k}})_{k\in I}$. For $k\in I$, we denote $\mb{C}_{k}:=\mb{C}^{\delta_{k}}$.

  For the sake of convenience, we will use a bold font like $\mbf{v},\mbf{w},$, etc. to denote the dimension vectors, and the capital letters like $V,W$, etc. to denote the $I$-graded vector space $\mb{C}^{\mbf{v}},\mb{C}^{\mbf{w}}$, etc.

\subsubsection{$E^{\#}$-graded homomorphisms and multiplication}
Let $Q=(I,E)$ be a quiver. Given two $I$-graded vector spaces $G,H$, we define
\begin{align*}
   E(G,H):= \bigoplus_{h\in E^{\#}}Hom(G_{in(h)},H_{out(h)}).
\end{align*}

\begin{example}
  Given $k\in I$, for any $I$-graded vector space $G$, we have
  \begin{align*}
  I(G,\mb{C}_{k})=G_{k}^{\vee},\quad E(G,\mb{C}_{k})=\oplus_{out(h)=k}G_{in(h)}^{\vee}, \\
  I(\mb{C}_{k},G)=G_{k},\quad E(\mb{C}_{k},G)=\oplus_{in(h)=k}G_{out(h)}.
\end{align*}
\end{example}
Given $B\in E(G,H)$ and $G'$ as an $I$-graded subspace of $G$, we say that $G'$ is $B$-invariant if for any $h\in E^{\#}$, $B_{h}G'_{in(h)}\subset G'_{out(h)}$. We define the multiplication in the following way:
\begin{enumerate}
\item for $B=(B_{h})\in E(G,H)$ and $C=(C_{h})\in E(H,K)$, we define
$$CB:=(\sum_{in(h)=k}C_{h}B_{\bar{h}})_{k}\in I(G,K);$$
\item for $B=(B_{k})\in I(G,H)$ and  $C=(C_{h})\in E(H,K)$, we define
$$CB:=(C_{h}B_{in(h)})_{h\in E^{\#}}\in E(G,K);$$
\item for $B=(B_{h})\in E(G,H)$ and  $C=(C_{k})\in I(H,K)$, we define
$$CB:=(C_{out(h)}B_{h})_{h\in E^{\#}}\in E(G,K);$$
\item for $B=(B_{k})\in I(G,H), C=(C_{k})\in I(H,V_{K})$, we define $$CB=(C_{k}B_{k})_{k\in I}\in I(G,K).$$
\end{enumerate}

\subsubsection{Derived schemes}
A derived scheme is a pair $X=(|X|,\mc{O}_{X})$,where $|X|$ is a topological space and $\mc{O}_{X}$ is a sheaf on $|X|$ with values in the $\infty$-category of simplicial commutative rings such that the ringed space $X^{cl}=(|X|,\pi_{0}(\mc{O}_{X}))$ is a scheme and the sheaf $\pi_{n}\mc{O}_{X}$ is a quasi-coherent $\pi_{0}\mc{O}_{X}$-module for each $n>0$. We will also denote $X^{cl}$ as $\pi_{0}(X)$, and say it as the underlying scheme of $X$. Here, all derived schemes will be defined over $\mb{C}$, hence derived schemes can be modeled
locally by dg-algebras rather than simplicial ones.

For a morphism of derived schemes $f:X\to Y$, we denote $L_{f}$ as the cotangent complex $f$. We define the tangent complex $T_{f}$ as the derived dual of $L_{f}$ and the conormal complex $C_{f}:=L_{f}[-1]$. We say that the morphism $f:X\to Y$ is quasi-smooth if $L_{f}$ is perfect with Tor-amplitude $[0,1]$ and constant rank. We say that the morphism $f:X\to Y$ is smooth if $L_{f}$ is a locally free sheaf of constant rank over $X$. Given a quasi-smooth morphism $f:X\to Y$, we define the virtual dimension $vdim(f):=rank(L_{f})$ and the canonical line bundle $K_{f}:=det(L_{f})$.
\begin{example}
When $f:X\to Y$ is a morphism of smooth schemes, we have
$$L_{f}\cong \{f^{*}T^{*}Y\to T^{*}X\},\quad T_{f}\cong \{TX\to f^{*}TY\}[-1]$$
and thus $f$ is quasi-smooth.
\end{example}
\begin{example}
    The composition of two quasi-smooth morphisms is also quasi-smooth.
\end{example}
We say that a derived scheme $X$ is classical if $\pi_{0}(X)\cong X$. In this paper, we always assume that $X$ is finite type over $Spec(\mb{C})$. Given a derived scheme $X$, we say $X$ is smooth or quasi-smooth if the canonical morphism to $Spec(\mb{C})$ is smooth or quasi-smooth.
Given a derived scheme $X$ and a global section of locally free sheaf $f:\mc{O}_{X}\to V$, we denote $\mb{R}\mc{Z}(f)$ as the derived zero locus of $f$. The closed embedding $\mb{R}\mc{Z}(f)\to X$ is always quasi-smooth.

\subsubsection{The derived projectivization and the derived blow up}
Let $X$ be a derived scheme and $F\in QCoh(X)$, we denote $\mb{P}_{X}(F)$ as the derived projectivization of Jiang \cite{jiang2022derived}, and $pr_{F}:\mb{P}_{X}(F)\to X$ as the projection morphisms. Over $\mb{P}_{X}(F)$, there is a universal line bundle which we denote as $\mc{O}_{\mb{P}_{X}(F)}(1)$.

Let $f:X\to Y$ be a closed embedding of derived schemes, we denote $\mb{B}l_{f}$ as the derived blow-up of $f$, and $pr_{f}:\mb{B}l_{f}\to Y$ as the projection morphism. We denote $\mc{O}_{\mb{B}l_{f}}(-1)$ as the virtual exceptional divisor of $\mb{B}l_{f}$.
\subsection{Acknowledgements} The author would like to thank Andrei Negu\c{t}, Qingyuan Jiang, and Yukinobu Toda for many helpful discussions, with special thanks to Jeroen Hekking for sharing his thesis with the author. The author is supported by World Premier International Research Center Initiative (WPI initiative), MEXT, Japan, and Grant-in-Aid for Scientific Research grant  (No. 22K13889) from JSPS Kakenhi, Japan.
\section{Quivers, and Representations of Preprojective Algebras}
\label{sec2}

In this section, we review the cohomology theory of preprojective algebra representations. In this section, we fix a quiver $Q=(I,E)$.

\subsection{The preprojective algebra and representations}
We fix a function $\epsilon: E^{\#}\to \mathbb{C}^{*}$ such that $\epsilon(h)+\epsilon(\bar{h})=0$ for any $h\in E$. Given $B\in E(G,H)$, we denote $\epsilon B:=\oplus_{h\in E^{\#}}\epsilon(h)B_{h}$. For any $B\in E(G,H)$ and $C\in E(H,K)$, we have $\epsilon BC+B\epsilon C=0$.

\begin{example}
  Given $k\in I$, we have
  $$E(\mb{C}_{k},\mb{C}_{k})=\oplus_{l\in E_{k}^{\#}}\mb{C}\cong \mb{C}^{2g_{k}}$$
  and we denote it as $\mbf{L}_{k}$. We denote $x=(x_{l})_{l\in E_{k}^{\#}}$ as an element of $\mbf{L}_{k}$. There is a symplectic form on $\mbf{L}_{k}$:
$$\omega_{k}(x^{1},x^{2}):= \epsilon(l)x_{l}^{1}x_{l}^{2}.$$
\end{example}

Let $kQ^{\#}$ be the path algebra of $Q^{\#}$, and we denote the preprojective algebra of $Q$ as $$\Pi_{Q}:=kQ^{\#}/\sum_{h\in E^{\#}}\epsilon(h)h\bar{h}.$$ 
 \begin{definition}
  We define an $I$-graded $\Pi_{Q}$-module as a pair $(H,B_{H})$, where $H$ is a finite dimensional $I$-graded vector space $H$ and $B_{H}\in E(H,H)$ such that $\epsilon B_{H}B_{H}=0$. 
 \end{definition}

In this paper, we will omit $B_{H}$ and just say $H$ is an $I$-graded $\Pi_{Q}$-module. We define the dimension of $H$ as its dimension as an $I$-graded vector space.
   Given two $I$-graded $\Pi_{Q}$-modules $G,H$, we define the $Hom$-space
   $$Hom_{\Pi_{Q}}(G,H)=\{\zeta\in I(G,H)|\zeta B_{G}=B_{H}\zeta\}.$$
   Given $\zeta\in Hom_{\Pi_{Q}}(G,H)$, we notice that $ker(\zeta)$, $im(\zeta)$ and $coker(\zeta)$ are also $I$-graded $\Pi_{Q}$-modules. Hence all the $I$-graded $\Pi_{Q}$-modules forms an abelian category.

\begin{example}
Given $k\in I$, for any $x\in \mbf{L}_{k}\cong E(\mb{C}_{k},\mb{C}_{k})$, we have $(\epsilon x)x=0$ and thus $x$ could be regarded as an $I$-graded $\Pi_{Q}$-module.
\end{example}
\subsection{$\mc{E}xt$-complex}
Given two $I$-graded $\Pi_{Q}$ modules $G,H$, we define the complex
$$\mc{E}xt^{\bullet}_{\Pi_{Q}}(G,H):=\{I(G,H)\xrightarrow{\sigma(G,H)}E(G,H)\xrightarrow{\tau(G,H)}I(G,H)\},$$
where
$$\sigma(G,H)(\zeta)=B_{H}\zeta-\zeta B_{G}, \quad \tau(G,H)(J)=\epsilon B_{H}J+\epsilon JB_{G}.$$
We denote
\begin{gather*}
   Ext^{0}_{\Pi_{Q}}(G,H):=ker(\sigma(G,H)),\quad Ext^{2}_{\Pi_{Q}}(G,H):=coker(\tau(G,H)), \\
  Ext^{1}_{\Pi_{Q}}(G,H):=ker(\tau(G,H))/im(\sigma(G,H)).
\end{gather*}
Then we have $Hom_{\Pi_{Q}}(G,H)\cong Ext^{0}_{\Pi_{Q}}(G,H)$.

Given $I$-graded $\Pi_{Q}$-modules $G,H,K$, the multiplication of $I$-graded homomorphisms naturally induces a morphism:
$$\mc{E}xt^{\bullet}_{\Pi_{Q}}(G,H)\otimes \mc{E}xt^{\bullet}_{\Pi_{Q}}(H,K)\to \mc{E}xt^{\bullet}_{\Pi_{Q}}(G,K),$$
and hence induces natural morphisms
$$\theta:Ext^{i}_{\Pi_{Q}}(G,H)\otimes  Ext^{j}_{\Pi_{Q}}(H,K)\to Ext^{i+j}_{\Pi_{Q}}(G,K).$$

The perfect pairing $tr_{\epsilon}$ between $E(G,H)$ and $E(H,G)$:
$$(J^{1},J^{2})\to tr(\epsilon J^{1}J^{2}).$$
induces an isomorphism $\tau(G,H)\cong \sigma(H,G)^{\vee}$ and hence we have the perfect pairing:
\begin{align*}
  \Omega(G,H):Ext^{l}_{\Pi_{Q}}(G,H)\otimes Ext^{2-l}_{\Pi_{Q}}(H,G)\to \mb{C},\quad   a\otimes b \to tr(\theta(a\otimes b)).
\end{align*}
We notice that $\Omega(G,H)(a\otimes b)=(-1)^{m(2-m)}\Omega(H,G)(b\otimes a)$, for $a\in Ext^{m}(G,H)$ and $b\in Ext^{2-m}(H,G)$. Moreover, $\Omega(G,G)$ is an symplectic form on $Ext^{1}(G,G)$ for any $I$-graded $\Pi_{Q}$-module $G$.

Given two dimension vectors $\mbf{v}^{1},\mbf{v}^{2}$, we introduce the bilinear pairing
\begin{gather*}
  \mbf{v}^{1}\bullet \mbf{v}^{2}:=\sum_{k\in I}\mbf{v}_{k}^{1}\mbf{v}_{k}^{2},\\
  <\mbf{v}^{1},\mbf{v}^{2}>_{Q}:=-\sum_{h\in E^{\#}}\mbf{v}_{in(h)}^{1}\mbf{v}_{out(h)}^{2}+2\mbf{v}^{1}\bullet \mbf{v}^{2}.
\end{gather*}
Then for any two $I$-graded $\Pi_{Q}$-modules $G,H$, we have
\begin{equation}
  \label{eq:2.1}
  \sum_{l=0}^{2}(-i)dim(Ext^{i}_{\Pi_{Q}}(G,H))=<dim(G),dim(H)>_{Q}.
\end{equation}
\subsection{Long exact sequences and extensions}
Let
$$ 0\to G\xrightarrow{\zeta_{GH}} H\xrightarrow{\zeta_{HK}} K\to 0$$
be a short exact sequence of $I$-graded $\Pi_{Q}$-modules. Then for another $I$-graded $\Pi_{Q}$-module $L$, we have the short exact sequence of complexes
\begin{align*}
  0\to \mc{E}xt^{\bullet}_{\Pi_{Q}}(L,G)\xrightarrow{\zeta_{GH}\circ -} \mc{E}xt^{\bullet}_{\Pi_{Q}}(L,H)\xrightarrow{\zeta_{HK}\circ -} \mc{E}xt^{\bullet}_{\Pi_{Q}}(L,K)\to 0 \\
   0\to \mc{E}xt^{\bullet}_{\Pi_{Q}}(K,L)\xrightarrow{-\circ \zeta_{HK}}\mc{E}xt^{\bullet}_{\Pi_{Q}}(H,L)\xrightarrow{-\circ\zeta_{GH}}\mc{E}xt^{\bullet}_{\Pi_{Q}}(G,L)\to 0.
\end{align*}
which induce the long exact sequences:
\begin{equation}
  \label{e011103}
      \begin{tikzcd}[column sep=12ex]
   0\ar{r}&   Ext^{0}(K,L)\ar{r}{\theta(\zeta_{HK}\otimes -)}  & Ext^{0}(H,L) \ar{r}{\theta(\zeta_{G H}\otimes -)} & Ext^{0}(G,L)   \\
       \ar{r}{\theta(\zeta_{KG}\otimes -)}&   Ext^{1}(K,L)\ar{r}{\theta(\zeta_{HK}\otimes -)} & Ext^{1}(H,L)\ar{r}{\theta(\zeta_{GH}\otimes -)} & Ext^{1}(G,L) & \\
       \ar{r}{\theta(\zeta_{KG}\otimes -)} &  Ext^{2}(K,L)\ar{r}{\theta(\zeta_{HK}\otimes -)} & Ext^{2}(H,L)\ar{r}{\theta(\zeta_{GH}\otimes -)} & Ext^{2}({B}^{1},L) & \\
       \ar{r} & 0
      \end{tikzcd}
    \end{equation}
    \begin{equation}
      \label{e011104}
      \begin{tikzcd}[column sep=12ex]
        0\ar{r} & Ext_{\Pi_{Q}}^{0}(L,G)\ar{r}{\theta(-\otimes \zeta_{GH})} & Ext_{\Pi_{Q}}^{0}(L,H)\ar{r}{\theta(-\otimes \zeta_{HK})} & Ext_{\Pi_{Q}}^{0}(L,K) &\\
        \ar{r}{\theta(-\otimes\zeta_{KG})} & Ext_{\Pi_{Q}}^{1}(L,G)\ar{r}{\theta(-\otimes \zeta_{GH})} & Ext_{\Pi_{Q}}^{1}(L,H)\ar{r}{\theta(-\otimes \zeta_{HK})} & Ext_{\Pi_{Q}}^{1}(L,K) &\\
        \ar{r}{\theta(-\otimes \zeta_{KG})} & Ext_{\Pi_{Q}}^{2}(L,G)\ar{r}{\theta(-\otimes \zeta_{GH})} & Ext_{\Pi_{Q}}^{2}(L,H)\ar{r}{\theta(-\otimes \zeta_{HK})} & Ext_{\Pi_{Q}}^{2}(L,K) \\
        \ar{r}& 0. 
      \end{tikzcd}
    \end{equation}
    Here $\zeta_{KG}\in Ext^{1}_{\Pi_{Q}}(K,G)$ is the image of $id\in Ext_{\Pi_{Q}}^{0}(G,G)$ in \cref{e011103} where we replace $L$ by $G$, and we denote $\zeta_{KG}$ as the extension class of the short exact sequence.

    \begin{definition}
      Given two $I$-graded $\Pi_{Q}$-modules $G,K$, an extension of $G$ and $K$ is a short exact sequence:
      $$C=\{0\to G\to H\to K\to 0\}$$
      We say two extensions $C$, $C'$ are scalar equivalent, if there exists a commutative diagram of $I$-graded $\Pi_{Q}$-modules:
      \begin{equation*}
        \begin{tikzcd}
          0\ar{r} & G\ar{r} \ar{d}{\lambda_{1}id}& H\ar{r} \ar{d}{\zeta} & K\ar{r}\ar{d}{\lambda_{3}id} & 0 \\
          0 \ar{r} & G\ar{r} & H^{'}\ar{r} & K\ar{r} & 0
        \end{tikzcd}
      \end{equation*}
      such that $\lambda_{1},\lambda_{3}\in \mb{C}^{*}$ and $\zeta$ is an isomorphism. We say $C$ and $C'$ are equivalent if we have the above diagram with $\lambda_{1}=\lambda_{3}=1$. We denote $\mc{M}_{\Pi_{Q}}^{ext}(G,K)$ (resp. $\mc{M}^{\mb{C}^{*}-ext}_{\Pi_{Q}}(G,K)$) the moduli space of extensions of $G$ and $K$ modulo the equivalent (resp. scalar equivalent) relations.
    \end{definition}
    \begin{lemma}
      \label{020916}
  The correspondence
  $$\{0\to G\to H\to K\to 0\} \to \zeta_{KG}$$
  induces isomorphisms:
  $$\mc{M}_{\Pi_{Q}}^{ext}(G,K)\cong Ext^{1}_{\Pi_{Q}}(K,G),\quad \mc{M}_{\Pi_{Q}}^{\mb{C}^{*}-ext}(G,K)\cong [Ext^{1}_{\Pi_{Q}}(K,G)/\mb{C}^{*}]$$
  where the $\mb{C}^{*}$ action is the scalar action. Moreover, the short exact sequence split if and only if $\zeta_{KG}=0$.
\end{lemma}
\begin{proof}
   For $C\in Ext^{1}_{\Pi_{Q}}(K,G)$, we want to construct a extension of $G$ and $K$ with extension class $C$. Let $C_{0}$ be a representative of $C$ in $E(K,G)$. Then $\tau(K,G)(C_{0})=0$. We notice that for $\zeta\in E(K,G)$ such that $\tau(K,G)(\zeta)=0$, the matrix
  $$B=
  \begin{pmatrix}
    B_{K}& C_{0} \\
    0 & B_{G}
  \end{pmatrix} \in E(K\oplus G)
  $$
  induces an $I$-graded $\Pi_{Q}$-module structure $H=(K\oplus G, B)$, with a short exact sequence
  $$0\to G\to H\to K\to 0.$$
  We notice that the above short exact sequence (up to equivalence) is independent of the choice $C_{0}$. It induces an inverse correspondence of $\zeta_{KG}$. The correspondence modulo the scalar equivalence follows from the correspondence modulo the equivalence.
\end{proof}

\subsection{Extensions and Lagragians}
Given $k\in I$, we say a sequence of $I$-graded $\Pi_{Q}$-modules $(G^{0},G^{1},\cdots G^{m})$ is a $k$-good sequence, if
\begin{enumerate}
\item \label{item1}For any $y\in \mbf{L}_{k}$, $Hom_{\Pi_{Q}}(y, G^{m})\cong 0$
\item \label{item2} For any $0\leq a\leq b\leq m$, there exists an injective morphism $\zeta_{G^{a}G^{b}}\in Hom_{\Pi_{Q}}(G^{a},G^{b})$ such that $Hom_{\Pi_{Q}}(G^{a},G^{b})\cong \mb{C}\zeta_{G^{a}G^{b}}$ and $\zeta_{G^{b}G^{c}}\circ \zeta_{G^{a}G^{b}}\cong \zeta_{G^{a}G^{c}}$ for all $0\leq a\leq b\leq c\leq m$.
\item  For $0\leq a<b\leq m$, $Hom_{\Pi_{Q}}(G^{b},G^{a})=0$.
\item For any $0\leq a<m$, $dim(G^{a})+\delta_{k}=dim(G^{a+1})$
\end{enumerate}
We notice that if $m\geq 1$, condition (\ref{item2}) induces that $\zeta_{G^{a}G^{a}}=id_{G^{a}}$ for all $0\leq a\leq m$.

Let $(G^{1},G^{2})$ be a $k$-good sequence, with the short exact sequence:
\begin{equation}
  \label{e021015}
0\to G^{1}\xrightarrow{\zeta_{G^{1}G^{2}}}G^{2}\xrightarrow{\zeta_{G^{2}x}} x\to 0,  
\end{equation}
where $x\in \mbf{L}_{k}$. By condition (\ref{item1}), \cref{e021015} does not split. We consider the following diagram of complexes:
\begin{equation}
  \begin{tikzcd}
    \label{diag:2.4}
   & 0\ar{d} & 0 \ar{d} & 0\ar{d} \\
 0 \ar{r}  & \mc{E}xt_{\Pi_{Q}}^{\bullet}(x, G^{1})\ar{r} \ar{d}& \mc{E}xt_{\Pi_{Q}}^{\bullet}(G^{2},G^{1}) \ar{r}\ar{d}& \mc{E}xt_{\Pi_{Q}}^{\bullet}(G^{1},G^{1}) \ar{d} \ar{r} & 0\\
 0\ar{r}  & \mc{E}xt_{\Pi_{Q}}^{\bullet}(x, G^{1}) \ar{d}\ar{r} & \mc{E}xt_{\Pi_{Q}}^{\bullet}(G^{2},G^{2}) \ar{r}\ar{d} & \mc{E}xt_{\Pi_{Q}}^{\bullet}(G^{1},G^{2})\ar{d} \ar{r} & 0 \\
   0\ar{r} &\mc{E}xt_{\Pi_{Q}}^{\bullet}(x,x)\ar{r} \ar{d}& \mc{E}xt_{\Pi_{Q}}^{\bullet}(G^{2},x)\ar{r}\ar{d} & \mc{E}xt_{\Pi_{Q}}^{\bullet}(G^{1},x) \ar{r}\ar{d} & 0\\
    & 0 & 0& 0 & 
  \end{tikzcd}
\end{equation}
where all the columns and rows are exact sequences of complexes. Hence the spectral sequence of the following double complex:
$$\{\mc{E}xt_{\Pi_{Q}}^{\bullet}(G^{2},G^{1})\to \mc{E}xt_{\Pi_{Q}}^{\bullet}(G^{1},G^{1})\oplus \mc{E}xt_{\Pi_{Q}}^{\bullet}(G^{2},G^{2})\to \mc{E}xt_{\Pi_{Q}}^{\bullet}(G^{1},G^{2})\}$$
converges to $\oplus_{i=0}^{2}Ext_{\Pi_{Q}}^{i}(x,x)[-i]$.
By computing the spectral sequence, we have the following proposition:
\begin{proposition}
    \label{011209}
   Consider the complex
  $$Ext^{1}_{\Pi_{Q}}(G^{2},G^{1})\xrightarrow{\sigma_{G^{1} G^{2}}}Ext^{1}_{\Pi_{Q}}(G^{1},G^{1})\oplus Ext^{1}_{\Pi_{Q}}(G^{2},G^{2})\xrightarrow{\tau_{G^{1} G^{2}}}Ext^{1}_{\Pi_{Q}}(G^{1},G^{2}),$$
  where
  \begin{align*}
    \sigma_{G^{1}G^{2}}=(\zeta_{G^{1},G^{2}}\otimes -)-(-\otimes \zeta_{G^{1}G^{2}}) ,\quad
    \tau_{G^{1}G^{2}}=(-\otimes \zeta_{G^{1}G^{2}})+(\zeta_{G^{1}G^{2}}\otimes -),
  \end{align*}
  are dual to each other under the symplectic form $\Omega(G^{1},G^{1})+\Omega(G^{2},G^{2})$.
  we have $\sigma_{G^{1}G^{2}}$ is injective and $\tau_{G^{1}G^{2}}$ is surjective. Moreover, let
 \begin{align*}
  T_{G^{1}G^{2}}:=ker(\tau_{G^{1}G^{2}}),\quad  T^{*}_{G^{1}G^{2}}:= coker(\sigma_{G^{1}G^{2}}).
 \end{align*}
 Then the above spectral sequence induces an isomorphism
\begin{align}
  \label{e011308}
 T_{G^{1}G^{2}}/ T^{*}_{G^{1}G^{2}}\cong Ext_{\Pi_{Q}}^{1}(x,x).
\end{align}
\end{proposition}

We notice that $\sigma_{G^{1}G^{2}}$ and $\tau_{G^{1}G^{2}}$ are dual to other under the perfect pairing $\Omega(G^{1})\oplus \Omega(G^{2})$. Let
$$dt_{G^{1}G^{2}}:T_{G^{1}G^{2}}\to Ext_{\Pi_{Q}}^{1}(x,x), \quad dt_{G^{1}G^{2}}^{*} : Ext_{\Pi_{Q}}^{1}(x,x)\to T_{G^{1}G^{2}}^{*}$$
be the morphism induced by  \cref{e011308}. Then we have the short exact sequence
\begin{equation}
  \label{e0204111}
0\to  T_{G^{1}G^{2}}\to Ext_{\Pi_{Q}}^{1}(G^{1},G^{1})\oplus Ext_{\Pi_{Q}}^{1}(G^{2},G^{2})\oplus Ext_{\Pi_{Q}}^{1}(x,x) \to  T_{G^{1}G^{2}}^{*}\to 0
\end{equation}
and thus $T_{G^{1}G^{2}}$ is a lagragian of $Ext_{\Pi_{Q}}^{1}(G^{1},G^{1})\oplus Ext_{\Pi_{Q}}^{1}(G^{2},G^{2})\oplus Ext_{\Pi_{Q}}^{1}(x,x)$.

By \cref{e011103} and \cref{e011104}, we have the diagrams where all the rows and columns are left in exact sequences:
\begin{equation}
  \label{e011107}
  \begin{tikzcd}
    & & 0\ar{d} & 0\ar{d} \\
    & 0 \ar{d} \ar{r} & Ext_{\Pi_{Q}}^{1}(x,G^{1})/\mb{C}\ar{d} \ar{r} & Ext_{\Pi_{Q}}^{1}(x, G^{2})\ar{d} \\
 0\ar{r}& Hom_{\Pi_{Q}}(G^{2},x)/\mb{C}\ar{r}\ar{d} & Ext_{\Pi_{Q}}^{1}(G^{2},G^{1})\ar{r}\ar{d} & Ext_{\Pi_{Q}}^{1}(G^{2},G^{2})\ar{d} \\
 0\ar{r}& Hom_{\Pi_{Q}}(G^{1},x)\ar{r}  & Ext_{\Pi_{Q}}^{1}(G^{1},G^{1})\ar{r}& Ext_{\Pi_{Q}}^{1}(G^{1},G^{2})
  \end{tikzcd}
\end{equation}

Let $\zeta_{xG^{1}}\in Ext_{\Pi_{Q}}^{1}(x,G^{1})$ be the extension class generated by the short exact sequence \cref{e021015}. For any $a\in Hom_{\Pi_{Q}}(G^{1},x)$, we consider $\theta(a\otimes \zeta_{xG^{1}})\in Ext_{\Pi_{Q}}^{1}(G^{1},G^{1})$ and $\theta(\zeta_{xG^{1}}\otimes a)\in Ext^{1}_{\Pi_{Q}}(x,x)$. By \cref{e011107} and \cref{diag:2.4}, $(\theta(a\otimes \zeta_{xG^{1}}),0)\in T_{G^{1}G^{2}}$ and
\begin{equation}
  \label{e021019}
  dt_{G^{1}G^{2}}(\theta(a\otimes \zeta_{xG^{1}}),0)=\theta(\zeta_{xG^{1}}\otimes a).
\end{equation}
Similarly, for $b\in Ext^{1}_{\Pi_{Q}}(x,G^{2})$, we have $(0,\theta(\zeta_{G^{2}x}\otimes b))\in T_{G^{1}G^{2}}$ and
\begin{equation}
  \label{e0210110}
   dt_{G^{1}G^{2}}(0, \theta(\zeta_{G^{2}x}\otimes b))=\theta(b\otimes \zeta_{G^{2}x}).
\end{equation}

\subsection{Nested triples} We fix $k\in I$ and a $k$-good sequence $(G^{0},G^{1},G^{2})$. Moreover, we assume there exists $x\in \mbf{L}_{k}$ with the following short exact sequences:
\begin{align}
  \label{e0210111}
  0\to G^{0}\xrightarrow{\zeta_{G^{0}G^{1}}}G^{1}\xrightarrow{\zeta_{G^{1}x}} x\to 0, \\
  \label{e0210112}
0\to G^{1}\xrightarrow{\zeta_{G^{1}G^{2}}}G^{2}\xrightarrow{\zeta_{G^{2}x}} x\to 0,   
\end{align}
Let $\zeta_{xG^{0}}\in Ext_{\Pi_{Q}}^{1}(x,G^{0})$ and $\zeta_{xG^{1}}\in Ext_{\Pi_{Q}}^{1}(x,G^{1})$ be the extension class of the short exact sequences \cref{e0210111} and \cref{e0210112} respectively. 

We consider $\alpha_{G^{1}}:=\theta(\zeta_{G^{1}x}\otimes \zeta_{xG^{1}})\in Ext^{1}_{\Pi_{Q}}(G^{1},G^{1})$ and $\alpha_{x}:=\theta(\zeta_{xG^{1}}\otimes \zeta_{G^{1}x})\in Ext^{1}_{\Pi_{Q}}(x,x)$. By \cref{e021019} and \cref{e0210110}, we have
\begin{equation}
  \label{e0210115}
  dt_{G^{1}G^{2}}(\alpha_{G^{1}},0)=\alpha_{x}=dt_{G^{0}G^{1}}(0,\alpha_{G^{1}})
\end{equation}

\begin{lemma}
  \label{010409} \label{010805}  \label{0204119} \label{0113011}
  We have $\alpha_{G^{1}}\neq 0$ and
  \begin{align}
     \label{e0210116}
        Ker(Ext_{\Pi_{Q}}^{1}(G^{1},G^{1})\to Ext_{\Pi_{Q}}^{1}(G^{0},G^{1})\oplus Ext_{\Pi_{Q}}^{1}(G^{1},G^{2}))=\mb{C}\alpha_{G^{1}} \\
         \label{e0210117}
   Ker(Ext_{\Pi_{Q}}^{1}(G^{1},G^{1})\oplus Ext_{\Pi_{Q}}^{1}(x,x)  \to T_{G^{0}G^{1}}^{*}\oplus T_{G^{1}G^{2}}^{*})=\mb{C}(\alpha_{G^{1}}\oplus \alpha_{x})
  \end{align}
\end{lemma}
\begin{proof}
  
 We have the diagram of left exact sequences
  \begin{equation*}
    \begin{tikzcd}
   & 0 \ar{d}  & 0\ar{d} & 0 \ar{d} \\
  0\ar{r} &   Ext_{\Pi_{Q}}^{0}(x,x)\ar{r}{\zeta_{xG^{1}}}\ar{d}{\zeta_{G^{1}x}} &  Ext_{\Pi_{Q}}^{1}(x,G^{1})\ar{r}\ar{d}{\theta(\zeta_{G^{1}x}\otimes -)} & Ext_{\Pi_{Q}}^{1}(x, G^{2})\ar{d} \\
  0\ar{r}   &  Ext_{\Pi_{Q}}^{0}(G^{1},x)\ar{r}{\theta(-\otimes \zeta_{xG^{1}})} \ar{d}& Ext_{\Pi_{Q}}^{1}(G^{1},G^{1})\ar{r} \ar{d}& Ext_{\Pi_{Q}}^{1}(G^{1},G^{2}). \\
  0 \ar{r} & Ext_{\Pi_{Q}}^{0}(G^{0},x)\ar{r} & Ext_{\Pi_{Q}}^{1}(G^{0},G^{1})
    \end{tikzcd}
  \end{equation*}
   As $\theta(\zeta_{G^{1}x}\otimes -)$ is a injective morphism, $\theta(\zeta_{G^{1}x}\otimes \zeta_{xG^{1}})=\alpha_{G^{1}}\neq 0$. By the above diagram,  $Ker(Ext_{\Pi_{Q}}^{1}(G^{1},G^{1})\to Ext_{\Pi_{Q}}^{1}(G^{0},G^{1})\oplus Ext_{\Pi_{Q}}^{1}(G^{1},G^{2}))$ is
   \begin{align*}
  im(\theta(\zeta_{G^{1}x}\otimes -)\cap im(\theta(-\otimes \zeta_{xG^{1}})).
   \end{align*}
   First we show that there $im(\theta(\zeta_{G^{1}x}\otimes -)\cap im(\theta(-\otimes \zeta_{xG^{1}}))=\mb{C}\alpha_{G^{1}}$.   Let $c\oplus d\in Ext^{0}(G^{1},x)\oplus Ext^{1}(x,G^{1})$ such that $\theta(c\otimes \zeta_{xG^{1}})=\theta(\zeta_{G^{1}x}\otimes d)$.  Then the image of $d$ in $Ext^{1}(x,G^{2})$ is $0$, and hence $d=\lambda \zeta_{xG^{1}}$ for some $\lambda\in \mb{C}$. Hence $\lambda \alpha_{G^{1}}=\theta(\zeta_{G^{1}x}\otimes d)$. Thus we prove \cref{e0210116}. Now we prove \cref{e0210117}. By \cref{e0210115}, we have
   $$(\alpha_{G^{1}}\oplus \alpha_{x})\in Ker(Ext_{\Pi_{Q}}^{1}(G^{1},G^{1})\oplus Ext_{\Pi_{Q}}^{1}(x,x)\to T_{G^{1}G^{2}}^{*}\oplus T_{G^{2}G^{3}}^{*}).$$
   Moreover, we have the short exact sequence:
  \begin{align*}
    0\to  Ext_{\Pi_{Q}}^{1}(x,x)\to  T_{G^{0} G^{1}}^{*}\to Ext_{\Pi_{Q}}^{1}(G^{0},G^{1})\to 0, \\
    0\to  Ext_{\Pi_{Q}}^{1}(x,x)\to  T_{G^{1} G^{2}}^{*}\to Ext_{\Pi_{Q}}^{1}(G^{1},G^{2})\to 0.
  \end{align*}
  Then if
  $$(a,b)\in Ker(Ext_{\Pi_{Q}}^{1}(G^{1},G^{1})\oplus Ext_{\Pi_{Q}}^{1}(x,x)\to  T_{G^{0} G^{1}}^{*}\oplus T_{G^{1}G^{2}}^{*}),$$
  we have $a\in ker(Ext_{\Pi_{Q}}^{1}(G^{1},G^{1})\to Ext_{\Pi_{Q}}^{1}(G^{0},G^{1})\oplus Ext_{\Pi_{Q}}^{1}(G^{1},G^{2}))$ and hence $a=\lambda \alpha_{G^{1}}$ for some $\lambda\in \mb{C}$. Thus $b-\lambda \alpha_{x}\in ker(Ext^{1}_{\Pi_{Q}}(x,x)\to T^{*}_{G^{0}G^{1}}\oplus T^{*}_{G^{1}G^{2}})$ and has to be $0$ since $dt^{*}_{G^{0}G^{1}}$ is injective.
\end{proof}

\subsection{Quadruples} We fix $k\in I$, and $I$-graded $\Pi_{Q}$-modules $G^{0},G^{1},G^{1'}$  and $G^{2}$ such that $(G^{0},G^{1},G^{2})$ and $(G^{0},G^{1'},G^{2})$ are $k$-good sequences. Moreover, we assume that we have the following commutative diagram of $I$-graded $\Pi_{Q}$ modules:
\begin{equation*}
  \begin{tikzcd}
    & 0 \ar{d} & 0\ar{d} & 0\ar{d} & \\
    0\ar{r} & G^{0}\ar{r}{\zeta_{G^{0}G^{1}}} \ar{d}{\zeta_{G^{0}G^{1'}}} & G^{1}\ar{r}{\zeta_{G^{1}x}} \ar{d}{\zeta_{G^{1}G^{2}}} & x \ar{r}\ar{d}{id} & 0 \\
        0\ar{r} & G^{1'}\ar{r}{\zeta_{G^{1'}G^{2}}}\ar{d}{\zeta_{G^{1'}y}} & G^{2}\ar{r}{\zeta_{G^{2}x}} \ar{d}{\zeta_{G^{2}y}} & x \ar{d} \ar{r} & 0\\
          0 \ar{r} & y\ar{r}{id} \ar{d} & y \ar{r}\ar{d} & 0 \\
          & 0 & 0 & 
        \end{tikzcd}
\end{equation*}
where all the rows and columns are short exact sequences and $x,y\in \mbf{L}_{k}$ for some $k\in I$.

\begin{lemma}
  \label{0124110}
  The following diagram
  $$Ext^{1}(G^{2},G^{1})\oplus Ext^{1}(G^{2},G^{1'})\to Ext^{1}(G^{1},G^{1})\oplus Ext^{1}(G^{1'},G^{1'})\oplus Ext^{1}(G^{2},G^{2})$$
  is injective.
\end{lemma}
\begin{proof}
  The short exact sequence:
$$0\to G^{0}\xrightarrow{\zeta_{G^{0}G^{2}}}G^{2}\xrightarrow{\zeta_{G^{2}x}\oplus \zeta_{G^{2}y}}x\oplus y\to 0.$$
induces the following left exact sequence
\begin{equation*}
  0\to Ext_{\Pi_{Q}}^{1}(x,G^{2})\oplus Ext_{\Pi_{Q}}^{1}(y,G^{2})\to Ext_{\Pi_{Q}}^{1}(G^{2},G^{2})\to Ext_{\Pi_{Q}}^{1}(G^{0},G^{2}).
\end{equation*}
  By \cref{e011107}, we have
  \begin{align*}
    ker(Ext_{\Pi_{Q}}(G^{2},G^{1})\to Ext_{\Pi_{Q}}^{1}(G^{1},G^{1}))\cong Ext_{\Pi_{Q}}^{1}(x,G^{1})/\mb{C} \\
    ker(Ext_{\Pi_{Q}}(G^{2},G^{1'})\to Ext_{\Pi_{Q}}^{1}(G^{1'},G^{1'}))\cong Ext_{\Pi_{Q}}^{1}(y,G^{1'})/\mb{C}
  \end{align*}
  and $Ext_{\Pi_{Q}}^{1}(x,G^{1})/\mb{C}$ (resp. $Ext_{\Pi_{Q}}^{1}(y,G^{1'})/\mb{C}$) is a subspace of $Ext_{\Pi_{Q}}^{1}(x,G^{2})$ (resp. $Ext_{\Pi_{Q}}^{1}(y,G^{1})$).
  Thus the morphism:
  $$Ext_{\Pi_{Q}}^{1}(x,G^{1})/\mb{C} \oplus Ext_{\Pi_{Q}}^{1}(y,G^{1'})/\mb{C}\to Ext_{\Pi_{Q}}^{1}(G^{2},G^{2})$$
  is also injective.
\end{proof}

\subsection{The moduli space of $I$-graded $\Pi_{Q}$-modules and the moment map}

Given a dimension vector $\mbf{v}$, let $\mc{M}_{Q}(\mbf{v})$ be the moduli space of $I$-graded $\Pi_{Q}$-modules with dimension $\mbf{v}$, modulo the isomorphic relation. We could represent $\mc{M}_{Q}(\mbf{v})$ in the following way: the algebraic group $G_{\mathbf{v}}:=\prod_{k\in I}G_{\mathbf{v}_{k}}$ acts on  $E(V,V)$ by
$$g\circ B\to gBg^{-1}.$$
 The Lie algebra of $G_{\mathbf{v}}$ is $I(V,V)$, which is identified with $I(V,V)^{\vee}$ by the bilinear form
\begin{align*}
  I(V,V)\times I(V,V)\to \mathbb{C}, \quad(B,B')\to Tr(BB').
\end{align*}

Under the symplectic form $\omega$ on $E(V,V)$
$$\omega(B,B'):=tr(\epsilon BB').$$
we have the moment map
\begin{equation}
  \label{e112701}
\mu:E(V,V)\to I(V,V), \quad \mu(B)=\epsilon BB.
\end{equation}
We define $Rep_{Q^{\#}}(\mbf{v}):=E(V,V)$ and define $\Pi_{Q}(\mbf{v}):=\mu^{-1}(0)$. Then we have
$$\mc{M}_{Q}(\mbf{v}):=[\Pi_{Q}(\mbf{v})/G_{\mbf{v}}].$$

For any $B\in \Pi_{Q}(\mbf{v})$, $(V,B)$ is an $I$-graded $\Pi_{Q}$-module, and we abuse the notation to still denote as $B$.

\section{Crawley-Boevey's Construction and Nakajima Quiver Varieties}
\label{sec3}
Let $\mbf{w}$ be a dimension vector of $Q$. Crawley-Boevey \cite{crawley-boevey_2001} associates a new quiver, which we denote as $Q_{\mbf{w}}$. The vertex set of $Q_{\mbf{w}}$ is $I_{\mbf{w}}:=I\cup \{\infty\}$, and the set of arrows is
$$E_{\mbf{w}}:=E\cup\{a_{i,j}|in(a_{i,j})=\infty, out(a_{i,j})=i,i\in I, j\in \{1,\cdots,\mbf{w}_{i}\}\},$$
i.e. adding $\mbf{w}_{i}$-many arrows from $\infty$ to $i$. We extend the function $\epsilon$ to $E_{\mbf{w}}^{\#}$ by setting $\epsilon(a_{ij})=1$. A dimension vector of $Q_{\mbf{w}}$ could be written as $(\mbf{v},\mbf{v}_{\infty})$, where $\mbf{v}$ is a dimension vector of $Q$ and $\mbf{v}_{\infty}\in \mb{Z}_{\geq 0}$.
\begin{example}
   When $\mbf{v}_{\infty}=0$, we have
   $$Rep_{Q_{\mbf{w}}^{\#}}(\mbf{v},0)=Rep_{Q^{\#}}(\mbf{v}),\quad \Pi_{Q_{\mbf{w}}}(\mbf{v},0)=\Pi_{Q}(\mbf{v}).$$
   Moreover, any $I$-graded $\Pi_{Q}$-module with dimension $\mbf{v}$ could be regard as a $I_{\mbf{w}}$-graded $\Pi_{Q_{\mbf{w}}}$-module with dimension $(\mbf{v},0)$.
 \end{example}

In this section, we review Crawley-Boevey's construction and introduce the Nakajima quiver variety.
\subsection{ADHM datum and stability condition}
\label{s020915}
When $\mbf{v}_{\infty}=1$, we have
  $$Rep_{Q_{\mbf{w}^{\#}}}=E(V,V)\oplus I(V,W)\oplus I(W,V),$$
  which we also denote as $\mbf{M}(\mbf{v},\mbf{w})$. An element $\mc{B}=(B,i,j)$ in $\mbf{M}(\mbf{v},\mbf{w})$ is called as an ADHM datum, where $B,i,j$ denote the above three components respectively. The algebraic group $G_{(\mbf{v},1)}=G_{\mbf{v}}\times \mb{C}^{*}$, where  $\mb{C}^{*}$ acts trivially on  $\mathbf{M}(\mathbf{v},\mathbf{w})$ and $G_{\mathbf{v}}$ acts on  $\mathbf{M}(\mathbf{v},\mathbf{w})$ by
$$g\circ (B,C,i,j)\to (gBg^{-1},gCg^{-1},gi, jg^{-1}).$$

The symplectic form $\omega(\mbf{v},\mbf{w})$ on $\mathbf{M}(\mathbf{v},\mathbf{w})$ is represented by
$$\omega(\mbf{v},\mbf{w})((B,i,j),(B',i',j')):=tr(\epsilon BB')+tr(ij'-i'j).$$ Moreover, we have the moment map 
$$\mbf{M}(\mbf{v},\mbf{w})\xrightarrow{\mu=\mu_{Q}\oplus \mu{\infty}}I(V,V)\oplus \mb{C},$$
where
$$\mu_{Q}(B,i,j)=\epsilon BB+ij,\quad \mu_{\infty}(B,i,j)=tr(\mu_{Q}(B,i,j)).$$
Hence we have
$$\Pi_{Q_{\mbf{w}}}(\mbf{v},1)=\mu^{-1}(0)=\mu_{Q}^{-1}(0).$$
For any $\mc{B}\in \Pi_{Q_{\mbf{w}}}(\mbf{v},1)$, We denote $[\mc{B}]$ as the geometric point of $\mc{B}$ in $[\Pi_{Q_{\mbf{w}}}(\mbf{v},1)/G_{\mbf{v}}]$, which is an isomorphic class of $I_{\mbf{w}}$-graded $\Pi_{Q_{\mbf{w}}}$-modules.

For any two $I_{\mbf{w}}$-graded $\Pi_{Q_{\mbf{w}}}$ modules $\mc{G},\mc{H}$, we denote
$$\mc{E}xt^{\bullet}(\mc{G},\mc{H}):=\mc{E}xt^{\bullet}_{\Pi_{Q_{\mbf{w}}}}(\mc{G},\mc{H}),\quad Ext^{l}(\mc{G},\mc{H}):=Ext^{l}_{\Pi_{Q_{\mbf{w}}}}(\mc{G},\mc{H}),$$
where $l=0,1,2$.
\subsection{Stability conditions}
We use the stability condition of Nakajima \cite{10.1215/S0012-7094-98-09120-7}:

\begin{definition}[Stability]
  Given an $I_{\mbf{w}}$-graded $\Pi_{Q_{\mbf{w}}}$ module $\mc{B}$ such that $\mc{B}\in \Pi_{Q_{\mbf{w}}}(\mbf{v},1)$, we define $\mc{B}$ to be stable, if for any $I$-graded $\Pi_{Q}$-module $H$,
  $$Hom_{\Pi_{Q_{\mbf{w}}}}(H,\mc{B})=0,$$
  i.e. $G$ does not have $I$-graded $\Pi_{Q}$-graded submodule.
\end{definition}

\begin{lemma}
  \label{020918}
  For any $I_{\mbf{w}}$-graded $\Pi_{Q_{\mbf{w}}}$ module $[\mc{B}]$ such that $\mc{B}\in \Pi_{Q_{\mbf{w}}}(\mbf{v},1)$, there exists a short exact sequence:
  $$0\to H\to \mc{B}\to \mc{C}\to 0$$
  such that $H$ is a $I$-graded $\Pi_{Q}$-module and $\mc{C}$ is stable.
\end{lemma}
\begin{proof}
  Let $H$ be the sum of all the $I$-graded $\Pi_{Q}$-submodules of $\mc{B}$. Then $H$ is a submodule of $\mc{B}$ such that $\mc{C}:=\mc{B}/H$ is stable.
\end{proof}

\subsection{$\mc{E}$xt complexes for stable objects}
\label{011405}

We choose two $I_{\mbf{w}}$-graded $\Pi_{Q_{\mbf{w}}}$-modules $\mc{B}^{1},\mc{B}^{2}$ such that $\mc{B}^{l}=(B^{l},i^{l},j^{l})\in \Pi_{Q_{\mbf{w}}}(\mbf{v}^{l},1)$, $l=1,2$. We consider the complex $\mathcal{T}(\mc{B}^{1},\mc{B}^{2})$:
$$I(V^{1},V^{2})\xrightarrow{\sigma(\mc{B}^{1},\mc{B}^{2})} E(V^{1},V^{2})\oplus I(W,V^{2})\oplus I(V^{1},W)\xrightarrow{\tau(\mc{B}^{1},\mc{B}^{2})}I(V^{1},V^{2})$$
where
\begin{align*}
  \sigma(\mc{B}^{1},\mc{B}^{2})(\zeta):=(B^{2}\zeta-\zeta B^{1})\oplus (-\zeta i^{1})\oplus j^{2}\zeta \\
  \tau(\mc{B}^{1},\mc{B}^{2})(J\oplus a \oplus b):=(\epsilon B^{2}J+\epsilon J B^{1}+i^{2}b+aj^{1}).
\end{align*}
We have $\sigma(\mc{B}^{1}, \mc{B}^{2})\cong \tau(\mc{B}^{2},\mc{B}^{1})^{\vee}$ under the perfect pairing $\omega(\mbf{v}^{1},\mbf{v}^{2})$ between $ E(V^{1},V^{2})\oplus I(W,V^{2})\oplus I(V^{1},W)$ and  $E(V^{2},V^{1})\oplus I(W,V^{1})\oplus I(V^{2},W)$:
$$((J^{1},a^{1},b^{1}),(J^{2},a^{2},b^{2}))\to Tr(\epsilon J^{1}J^{2})+Tr(b^{2}a^{1}-b^{1}a^{2}).$$
\begin{lemma}[Proposition 3.5 of \cite{McGerty2018}, and Lemma 5.2 of \cite{10.1215/S0012-7094-98-09120-7} when $Q$ has no loops]
  \label{112302}
  If $\mc{B}^{2}$ is stable, then $\sigma(\mc{B}^{1}, \mc{B}^{2})$ is injective. If $\mc{B}^{1}$ is stable, then $\tau(\mc{B}^{1}, \mc{B}^{2})$ is surjective.
\end{lemma}
\begin{proof}
  By the duality, we only need to prove that if $\mc{B}^{2}$ is stable, then $\sigma(\mc{B}^{1}, \mc{B}^{2})$ is injective. We just need to notice that for $\zeta\in ker(\sigma(\mc{B}^{1},\mc{B}^{2}))$, $(\zeta,0)\in Hom(\mc{B}^{1},\mc{B}^{2})$ and its image is a $I$-graded $\Pi_{Q}$-module, and thus has to be $0$.
\end{proof}

Now, we assume $\mc{B}^{1}$ and $\mc{B}^{2}$ in this subsection are both stable. We denote $T(\mc{B}^{1},\mc{B}^{2}):=ker(\tau(\mc{B}^{1},\mc{B}^{2}))$. The perfect pairing $\omega(\mbf{v}^{1},\mbf{v}^{2})$ descends to a perfect pairing:
$$\Omega(\mc{B}^{1},\mc{B}^{2}):T(\mc{B}^{1},\mc{B}^{2})\otimes T(\mc{B}^{2},\mc{B}^{1})\to \mb{C}$$

We consider the complex
\begin{align*}
  \mb{C}\xrightarrow{s^{1}(\mc{B}^{1}, \mc{B}^{2})} T(\mc{B}^{1},\mc{B}^{2}) \xrightarrow{s^{2}(\mc{B}^{1}, \mc{B}^{2})} \mb{C},
\end{align*}
where
\begin{align*}
  s^{1}(\mc{B}^{1}, \mc{B}^{2})(t):=t(0\oplus (-i_{2})\oplus j^{1}) \quad (mod \quad im(\sigma(\mc{B}^{1},\mc{B}^{2})))\\
  s^{2}(\mc{B}^{1}, \mc{B}^{2})(J\oplus a\oplus b)\quad  (mod \quad im(\sigma(\mc{B}^{1},\mc{B}^{2}))):=tr(i^{1}b+aj^{2}).
\end{align*}
Then $s^{1}(\mc{B}^{1},\mc{B}^{2})$ and $s^{2}(\mc{B}^{2},\mc{B}^{1})$ are dual to each other, under the perfect pairing  we have
\begin{gather*}
  Ext^{0}(\mc{B}^{1},\mc{B}^{2})=ker(s^{1}(\mc{B}^{1},\mc{B}^{2})),\quad Ext^{2}(\mc{B}^{1},\mc{B}^{2})=coker(s^{2}(\mc{B}^{1},\mc{B}^{2}))\\
  Ext^{1}(\mc{B}^{0},\mc{B}^{1})=ker(s^{2}(\mc{B}^{1},\mc{B}^{2}))/im(s^{1}(\mc{B}^{1},\mc{B}^{2}))  
\end{gather*}

\begin{definition}
  We define $\mc{B}^{1}\subset \mc{B}^{2}$, if $s^{1}(\mc{B}^{1},\mc{B}^{2})=0$ ,i.e. there exists an $I$-graded morphism $\zeta\in I(V^{1},V^{2})$ such that
  \begin{equation}
    \label{e112301}
    \zeta B^{1}=B^{2}\zeta, \quad \zeta i^{1}=i^{2} \quad j^{1}=j^{2}\zeta,
  \end{equation}
\end{definition}
\begin{lemma}
  \label{010802}
  If $\mc{B}^{1}$ and $\mc{B}^{2}$ are both stable, then
  \begin{equation*}
    Ext^{0}(\mc{B}^{1},\mc{B}^{2})\cong
    \begin{cases}
      \mb{C} & \text{ if } \mc{B}^{1}\subset \mc{B}^{2} \\
      0 & \text{ otherwise. }
    \end{cases}
  \end{equation*}
  Moreover, $\zeta \in I(V^{1},V^{2})$ which satisfy \cref{e112301} is unique and injective as an $I_{\mbf{w}}$-graded morphism of vector spaces.
\end{lemma}
By \cref{010802}, $\mc{B}^{1}\subset \mc{B}^{2}$ if and only if $\mc{B}^{1}$ is a submodule of $\mc{B}^{2}$. Moreover, if $\mc{B}^{1}\subset \mc{B}^{2}$, we have
\begin{equation*}
  Ext^{1}(\mc{B}^{1},\mc{B}^{2})\cong
  \begin{cases}
    T(\mc{B}^{1},\mc{B}^{2}) & \text{ if } \mc{B}^{1}\cong \mc{B}^{2}\\
    ker(s^{2}(\mc{B}^{1},\mc{B}^{2})) & \text{ otherwise.}
  \end{cases}
\end{equation*}

\subsection{The universal complex}

First, we recall the fact $\mbf{L}_{k}=\Pi_{Q_{\mbf{w}}}(\delta_{k},0)$. For $x,y\in\mbf{L}_{k}$, the $\mc{E}$xt complex $\mc{E}xt^{\bullet}(x,y)$ is represented by
$$\mc{E}xt^{\bullet}(x,y):=\{\mb{C}\xrightarrow{\sigma(x,y)}\mbf{L}_{k}\xrightarrow{\tau(x,y)}\mb{C}\}$$
where $\sigma(x,y)(\lambda)=\lambda(x-y)$ and $\tau(x,y)(z)=\sum_{l\in E_{k}^{\#}}\epsilon(l)z_{l}(x_{l}-y_{l})$. When $x\neq y$, $\sigma(x,y)$ is injective and $\tau(x,y)$ is surjective. When $x=y$, we have $\mc{E}xt^{\bullet}(x,y)\cong\mb{C}\oplus \mbf{L}_{k}[-1]\oplus \mb{C}[-2]$.

Secondly, we consider the case that $x\in \mbf{L}_{k}$ and $\mc{B}=(B,i,j)\in \Pi_{Q_{\mbf{w}}}(\mbf{v},1)$. Given $B\in E(V,V)$, we denote $(B-x)\in E(V,V)$ such that
\begin{equation*}
  (B-x)_{h}=
  \begin{cases}
    B_{h}-x_{h} & in(h)=out(h)=k, \\
    B_{h} & \text{otherwise.}
  \end{cases}
\end{equation*}

The $\mc{E}xt$ complex $\mc{E}xt^{\bullet}(x,\mc{B})$ is represented by:

$$I(\mb{C}_{k},V)\xrightarrow{\sigma_{k}(x,\mc{B})}E(\mb{C}_{k},V)\oplus I(\mb{C}_{k},W)\xrightarrow{\tau_{k}(x,\mc{B})}I(\mb{C}_{k},V)$$
where
\begin{align*}
  \sigma_{k}(x,\mc{B})(a):=(B-x)a\oplus j_{k}a, \quad  \tau_{k}(x,\mc{B})(C\oplus D):=\epsilon (B-x)C+i_{k}D.
\end{align*}

We have the formula:
\begin{equation}
  \sum_{l=0}^{2}(-1)^{l}(dim(Ext^{l}(x,\mc{B})))=\sum_{l=0}^{2}(-1)^{l}(dim(Ext^{l}(\mc{B},x)))=<\delta_{k},\mbf{v}>_{Q}-\mbf{w}_{k}.
\end{equation}
The complex $\mc{E}xt^{\bullet}(x)(\mc{B},x)$ is represented by:
$$I(V,\mb{C}_{k})\xrightarrow{\sigma_{k}(\mc{B},x)}E(V,\mb{C}_{k})\oplus I(W,\mb{C}_{k})\xrightarrow{\tau_{k}(\mc{B},x)}I(V,\mb{C}_{k})$$
where
$$\tau_{k}(\mc{B},x)a:=a\epsilon (B-x)\oplus ai_{k},\quad \sigma_{k}(\mc{B},x)(C\oplus D):=C(B-x)+Cj_{k}.$$

The following lemma follows from the definition of stability and duality:
\begin{lemma}
  \label{020921}
  If $\mc{B}$ is stable, then  $Ext^{0}(x,\mc{B})=0$ and $Ext^{2}(\mc{B},x)=0$. Hence $\sigma_{k}(x,\mc{B})$ is injective and $\tau_{k}(\mc{B},x)$ is surjective. 
\end{lemma}

\begin{definition}
  If $\mc{B}$ is stable, we define the universal complex
  $$\mc{I}(x,\mc{B}):=\{U(\mb{C}_{k},V)\xrightarrow{u_{k}(x,\mc{B})}I(\mb{C}_{k},V)\}$$
  such that $U(\mb{C}_{k},V)=ker(\sigma_{k}(\mc{B},x))$ and $u_{k}(x,\mc{B})$ is the morphism induced from $\tau_{k}(\mc{B},x)$. 
\end{definition}
We notice that
\begin{equation}
  \label{eq:3.10}
ker(u_{k}(x,\mc{B}))=Ext^{1}(x,\mc{B}) \quad coker(u_{k}(x,\mc{B}))=Ext^{2}(x,\mc{B}).
\end{equation}

\subsection{Nakajima quiver variety} Now we introduce the definition of Nakajima quiver variety. Let 
$$\Pi_{Q_{\mbf{w}}}^{s}(\mbf{v},1):=\{\mc{B}\in \Pi_{Q_{\mbf{w}}}(\mbf{v},1)|\mc{B} \text{ is stable}.\}$$

\begin{definition}[Nakajima quiver variety]
  We define the Nakajima quiver variety as
  $$\mathfrak{M}(\mathbf{v},\mathbf{w}):=[\Pi_{Q_{\mbf{w}}}^{s}(\mbf{v},1)/G_{\mbf{v}}]$$
  For a closed point $\mc{B}=(B,i,j)\in \Pi_{Q_{\mbf{w}}}^{s}(\mbf{v},1)$, its $G$-orbit $[\mc{B}]$ is a geometric point in $\mathfrak{M}(\mathbf{v},\mathbf{w})$, and we also denote as $[B,i,j]$.
\end{definition}

\begin{theorem}[Nakajima \cite{10.1215/S0012-7094-98-09120-7}]
  The variety  $\Pi_{Q_{\mbf{w}}}^{s}(\mbf{v},1)$ is smooth and the $G_{\mbf{v}}$ action on $\Pi_{Q_{\mbf{w}}}^{s}(\mbf{v},1)$ is free. Moreover, the quotient $\mf{M}(\mbf{v},\mbf{w})$ is a smooth variety. For any closed point $[\mc{B}]\in \mf{M}(\mbf{v},\mbf{w})$, the tangent space at $[\mc{B}]$ is $T(\mc{B},\mc{B})$. Its dimension is $<\mbf{v},\mbf{v}>_{Q}+2\mbf{v}\bullet \mbf{w}$.
\end{theorem}

\subsection{Nested Pairs}
\label{s020317}

\begin{lemma}
  \label{020926}
  Given a short exact sequence of $I_{\mbf{w}}$-graded $\Pi_{Q_{\mbf{w}}}$-modules
  \begin{equation}
    \label{e020923}
  0\to \mc{B}^{1}\xrightarrow{\zeta_{\mc{B}^{1}\mc{B}^{2}}} \mc{B}^{2}\xrightarrow{\zeta_{\mc{B}^{2}x}} x\to 0    
  \end{equation}
  where $\mc{B}^{i}\in \Pi_{Q_{\mbf{w}}}^{s}(\mbf{v}^{l},1)$, $l=1,2$, and $x\in \mbf{L}_{k}$. If the short exact sequence does not split, then $\mc{B}^{2}$ is stable if and only if $\mc{B}^{1}$ is stable.
\end{lemma}
\begin{proof}
  If $\mc{B}^{2}$ is stable, as a submodule $\mc{B}^{1}$ is also stable. On the other hand, if $\mc{B}^{1}$ is stable, we consider the short exact sequence
  $$0\to H\xrightarrow{a} \mc{B}^{2}\to \mc{C}\to 0$$
  by \cref{020918}, such that $\mc{C}$ is stable and $H$ is a $\Pi_{Q}$-module. Then $\zeta_{\mc{B}^{2}x}\circ a$ has to be injective by the stability of $\mc{B}^{1}$. As the short exact sequence \cref{e020923} does not split. Then $H$ has to be $0$ and thus $\mc{B}^{2}$ is stable.
\end{proof}

We notice that if we have a short exact sequence \cref{e020923}, then we have
\begin{equation}
  \label{e020924}
   x_{l}=tr(B_{l}^{2})-tr(B_{l}^{1}),\quad \forall l\in E_{k}^{\#}.
 \end{equation}
We denote $\mc{B}^{1}\subset_{x}\mc{B}^{2}$ if we have a short exact sequence \cref{e020923}. We notice that the condition $\mc{B}^{1}\subset \mc{B}^{2}$ or $\mc{B}^{1}\subset_{x}\mc{B}^{2}$ is independent of the $G_{\mbf{v}^{1}}$ or $G_{\mbf{v}^{2}}$ action on $\mc{B}^{1}$ or $\mc{B}^{2}$. Hence we have the following definition:
\begin{definition}
  Given two closed points $[\mc{B}^{l}]\in \mf{M}(\mbf{v}^{l},\mbf{w}), l=1,2$, we define $[\mc{B}^{1}]\subset [\mc{B}^{2}]$ (resp. $[\mc{B}^{1}]\subset_{x} [\mc{B}^{2}]$ ) if $\mc{B}^{1}\subset \mc{B}^{2}$ (resp. $\mc{B}^{1}\subset_{x}\mc{B}^{2}$).
\end{definition}

In the last of this subsection, we fix two dimension vectors $\mbf{v}^{2}=\mbf{v}^{1}+\delta_{k}$.
Given $[\mc{B}^{1}]\subset_{x}[\mc{B}^{2}]$, by \cref{010802} we have $Hom(\mc{B}^{1},\mc{B}^{2})\cong \mb{C}$. Hence all the short exact sequences of the form \cref{e020924} are scalar equivalent. Thus the morphism $\zeta_{\mc{B}^{1}\mc{B}^{2}}$ and $\zeta_{\mc{B}^{2}x}$ are unique up to scalar. Moreover, \cref{e020924} also decides an element $\zeta_{x\mc{B}^{1}}\in Ext^{1}(x,\mc{B}^{1})$ by \cref{020916} up to scalar equivalence.
\begin{lemma}
  \label{012727}
  Given a closed point $([\mc{B}^{2}],x)\in \mf{M}(\mbf{v}^{2},\mbf{w})\times \mbf{L}_{k}$. Then the correspondence
  $$[\mc{B}^{1}]\subset_{x}[\mc{B}^{2}]\to \zeta_{\mc{B}^{1}x}$$
  is a one-to-one correspondence between the pair $[\mc{B}^{1}]\subset_{x}[\mc{B}^{2}]$, where $[\mc{B}^{1}]\subset \mf{M}(\mbf{v}^{1},\mbf{w})$ and $1$-dimensional subspace of $Ext^{0}(\mc{B}^{2},x)=ker(u_{k}(x,\mc{B})^{\vee})$. 
\end{lemma}
\begin{proof}
  The inverse correspondence is constructed in the following way: given a non-zero vector $\zeta_{\mc{B}^{2}x}\in Ext^{0}(\mc{B}^{2},x)$, we regard $\zeta_{\mc{B}^{2}x}$ as a homomorphism of $I_{\mbf{w}}$-graded $\Pi_{Q_{\mbf{w}}}$-modules between $[\mc{B}^{2}]$ and $x$. Then $\zeta_{\mc{B}^{2},x}$ is surjective, and let $[\mc{B}^{1}]$ be its kernel. By \cref{020926}, $[\mc{B}^{1}]$ is also stable.
\end{proof}

\begin{lemma}
  \label{012728}
  Given a closed point $(\mc{B}^{1},x)\in \mf{M}(\mbf{v}^{1},\mbf{w})\times \mbf{L}_{k}$, the correspondence
$$[\mc{B}^{1}]\subset_{x}[\mc{B}^{2}]\to \zeta_{x\mc{B}^{1}}$$ is a 
  one-to-one correspondence between the pair $([\mc{B}^{1}]\subset_{x}[\mc{B}^{2}])$ where $[\mc{B}^{2}]\in \mf{M}(\mbf{v}^{2},\mbf{w})$ and one dimensional subspace $Ext^{1}(x,\mc{B}^{1})=ker(u_{k}(x,\mc{B}))$.
\end{lemma}
\begin{proof}It follows from \cref{020916} and \cref{020926}.
\end{proof}

\section{Hecke Correspondence and Tautological Bundles}
\label{sec4}
In this section, we fix a quiver $Q=(I,E)$ and a dimension vector $\mbf{w}$. Nakajima \cite{10.1215/S0012-7094-98-09120-7}\cite{10.2307/2646205} introduced the Hecke correspondence:

\begin{definition}[Hecke correspondence]
\label{def:4.1}
  Given two dimension vectors $\mbf{v}^{1},\mbf{v}^{2}$, we define the hecke correspondence
$$\mf{P}(\mbf{v}^{1},\mbf{v}^{2}):=\{[\mc{B}^{1}]\subset [\mc{B}^{2}]|[\mc{B}^{l}]\in \mf{M}(\mbf{v}^{l},\mbf{w}),l=1,2\}. $$
\end{definition}

The projections to $[\mc{B}^{1}]$ and $[\mc{B}^{2}]$ respectively induce morphisms:
$$p(\mbf{v}^{1},\mbf{v}^{2}):\mf{P}(\mbf{v}^{1},\mbf{v}^{2})\to\mf{M}(\mbf{v}^{1},\mbf{w}),\quad q(\mbf{v}^{1},\mbf{v}^{2}):\mf{P}(\mbf{v}^{1},\mbf{v}^{2})\to \mf{M}(\mbf{v}^{2},\mbf{w}).$$
\begin{example}
  When $\mbf{v}^{1}=\mbf{v}^{2}$, then $\mf{P}(\mbf{v}^{1},\mbf{v}^{2})\cong \mf{M}(\mbf{v},\mbf{w})$ and $p\times q(\mbf{v}^{1},\mbf{v}^{2})$ is the diagonal embedding of $\mf{M}(\mbf{v},\mbf{w})$.
\end{example}

\begin{example}
  When $\mbf{v}^{1}_{k}<\mbf{v}^{2}_{k}$, by \cref{112302}, $\mf{P}(\mbf{v}^{1},\mbf{v}^{2})=\emptyset$.
\end{example}
In this section, we will study the geometry of the Hecke correspondence.

\subsection{Tautological bundles on quiver varieties}

For any $k\in I$, the action of $G_{\mathbf{v}}$ on $\mathbf{M}(\mathbf{v},\mbf{w})$ makes $V_{k}$ and $W_{k}$ equivariant locally free sheaves, which descent to locally free sheaves $\mathfrak{M}(\mathbf{v},\mbf{w})$ and we still denote as $V_{k}$ and $W_{k}$.  We also regard $I(V,W)$, $I(W,V)$, $E(V,V)$, $L(V,V)$ as locally free sheaves on $\mathfrak{M}(\mathbf{v},\mbf{w})$. 

The $\mc{E}$xt complexes  in \cref{011405} could also be generalized as complexes of tautological locally free sheaves on quiver varieties: we fix dimension vectors $\mbf{v}^{1},\mbf{v}^{2}$ and denote by $V_{k}^{1}$ (resp. $V_{k}^{2}$) the locally free sheaf $V_{k}\boxtimes \mathcal{O}_{\mathfrak{M}}$ (resp. $\mathcal{O}_{\mathfrak{M}}\boxtimes V_{k}$) on $\mathfrak{M}(\mbf{v}^{1},\mbf{w})\times \mathfrak{M}(\mbf{v}^{2},\mbf{w})$. We consider the complex $\mathcal{T}(\mathbf{v}^{1},\mbf{v}^{2})$:
$$I(V^{1},V^{2})\xrightarrow{\sigma(\mbf{v}^{1},\mbf{v}^{2})} E(V^{1},V^{2})\oplus I(W,V^{2})\oplus I(V^{1},W)\xrightarrow{\tau(\mbf{v}^{1},\mbf{v}^{2})}I(V^{1},V^{2})$$
\begin{align*}
  \sigma(\mbf{v}^{1},\mbf{v}^{2})(\zeta):=(B^{2}\zeta-\zeta B^{1})\oplus (-\zeta i^{1})\oplus j^{2}\zeta \\
  \tau(\mbf{v}^{1},\mbf{v}^{2})(J\oplus a \oplus b):=(\epsilon B^{2}J+\epsilon J B^{1}+i^{2}b+aj^{1}).
\end{align*}
For any closed point $([\mc{B}^{1}],[\mc{B}^{2}])\in \mathfrak{M}(\mbf{v}^{1},\mbf{w})\times \mathfrak{M}(\mbf{v}^{2},\mbf{w})$,
$$\mathcal{T}(\mathbf{v}^{1},\mbf{v}^{2})|_{([\mc{B}^{1}],[\mc{B}^{2}])}\cong \mathcal{T}(\mc{B}^{1},\mc{B}^{2}).$$
By \cref{112302}, $ker(\tau(\mbf{v}^{1},\mbf{v}^{2}))/im(\sigma(\mbf{v}^{1},\mbf{v}^{2}))$ is still locally free on $ \mathfrak{M}(\mbf{v}^{1},\mbf{w})\times \mathfrak{M}(\mbf{v}^{2},\mbf{w})$, which we denote as $T(\mbf{v}^{1},\mbf{v}^{2})$. Under the perfect pairing $\omega(\mbf{v}^{1},\mbf{v}^{2}):$
\begin{align*}
  (E(V^{1},V^{2})\oplus I(W,V^{2})\oplus I(V^{1},W))\otimes (E(V^{2},V^{1})\oplus I(W,V^{1})\oplus I(V^{2},W))\to \mc{O} \\
  ((J^{1},a^{1},b^{1}),(J^{2},a^{2},b^{2}))\to Tr(\epsilon J^{1}J^{2})+Tr(b^{2}a^{1}-b^{1}a^{2}).
\end{align*}
we have  $\sigma(\mathbf{v}^{1},\mathbf{v}^{2})^{\vee}\cong \tau(\mbf{v}^{2},\mbf{v}^{1})$. Hence $\omega(\mbf{v}^{1},\mbf{v}^{2})$ descends to a perefect pairing between $T(\mbf{v}^{1},\mbf{v}^{2})$ and $T(\mbf{v}^{2},\mbf{v}^{1})$, which we denote as $\Omega(\mbf{v}^{1},\mbf{v}^{2})$.

We consider the complex 
 $$\mc{O}\xrightarrow{s^{1}(\mbf{v}^{1},\mbf{v}^{2})} \mathcal{T}(\mathbf{v}^{1},\mbf{v}^{2})\xrightarrow{s^{2}(\mbf{v}^{1},\mbf{v}^{2})}\mc{O},$$
where 
\begin{align*}
  s^{1}(\mbf{v}^{1},\mbf{v}^{2}):\mathcal{O}\to \mathcal{T}(\mathbf{v}^{1},\mbf{v}^{2}) \quad s^{1}(t):=t(0\oplus (-i^{2})\oplus j^{1}) \text{ mod } Im(\sigma(\mbf{v}^{1},\mbf{v}^{2})) \\
  s^{2}(\mbf{v}^{1},\mbf{v}^{2}) \text{ mod } Im(\sigma(\mbf{v}^{1},\mbf{v}^{2})):\mathcal{T}(\mathbf{v}^{1},\mbf{v}^{2})\to \mathcal{O} \quad s^{2}(J\oplus a\oplus b):=tr(i^{1}b+aj^{2}).
\end{align*}
Then for any closed point $([\mc{B}^{1}],[\mc{B}^{2}])\in \mathfrak{M}(\mbf{v}^{1},\mbf{w})\times \mathfrak{M}(\mbf{v}^{2},\mbf{w})$,
$$s^{l}(\mbf{v}^{1},\mbf{v}^{2})|_{([\mc{B}^{1}],[\mc{B}^{2}])}\cong s^{l}(\mc{B}^{1},\mc{B}^{2}),$$ where $l=1,2$.  We also have $s^{2}(\mbf{v}^{2},\mbf{v}^{1})\cong s^{1}(\mbf{v}^{1},\mbf{v}^{2})^{\vee}$ under the perfect pairing $\Omega(\mbf{v}^{1},\mbf{v}^{2})$.

\begin{theorem}
  \label{020934}
    The Hecke correspondence $\mf{P}(\mbf{v}^{1},\mbf{v}^{2})$ is the zero locus of $s^{1}(\mbf{v}^{1},\mbf{v}^{2})$. The rank of $T(\mbf{v}^{1},\mbf{v}^{2})$ is $-<\mbf{v}^{1},\mbf{v}^{2}>_{Q}+(\mbf{v}^{1}+\mbf{v}^{2})\bullet\mbf{w}$.
\end{theorem}
\begin{proof}
  The only thing we still need to prove is that $\mf{P}(\mbf{v}^{1},\mbf{v}^{2})$ is the zero locus of $s^{1}(\mbf{v}^{1},\mbf{v}^{2})$, which follows from  \cref{010802}.
\end{proof}

Restricting to $\mf{P}(\mbf{v}^{1},\mbf{v}^{2})$, there is a unique $I$-graded morphism $\zeta_{\mbf{v}^{1},\mbf{v}^{2}}:V^{1}\to V^{2}$ such that
\begin{equation*}
    \zeta_{\mbf{v}^{1},\mbf{v}^{2}} B^{1}=B^{2}\zeta_{\mbf{v}^{1},\mbf{v}^{2}}, \quad \zeta_{\mbf{v}^{1},\mbf{v}^{2}} i^{1}=i^{2} \quad j^{1}=j^{2}\zeta_{\mbf{v}^{1},\mbf{v}^{2}}.
  \end{equation*}
  By \cref{112302}, $\zeta_{\mbf{v}^{1},\mbf{v}^{2}}$ is injective at every closed point.

\subsection{The case that $\mbf{v}=\mbf{v}^{1}=\mbf{v}^{2}$}

 First we consider the case $\mbf{v}=\mbf{v}^{1}=\mbf{v}^{2}$. By \cref{112302}, the diagonal of $\mf{M}(\mbf{v})\times \mf{M}(\mbf{v})$, which we denote as $\Delta_{\mf{M}(\mbf{v})}$, is the zero locus of $s^{1}(\mathbf{v},\mathbf{v})=s^{2}(\mathbf{v},\mathbf{v})^{\vee}$. Moreover, as $dim(T(\mbf{v},\mbf{v}))=dim(\mf{M}(\mbf{v},\mbf{w}))$, we have
 \begin{theorem}
   \label{thm35}
    We consider the following complex $T(\mathbf{v}):=T(\mbf{v},\mbf{v})|_{\Delta_{\mf{M}(\mbf{v},\mbf{w})}}$
$$I(V,V)\xrightarrow{\sigma(\mbf{v})} E(V,V)\oplus I(W,V)\oplus I(V,W)\xrightarrow{\tau(\mbf{v})}I(V,V)$$
where
\begin{align*}
  \sigma(\mbf{v})(\zeta):=(B\zeta-\zeta B)\oplus (-\zeta i)\oplus j\zeta \qquad
  \tau(\mbf{v})(J\oplus a \oplus b):=(\epsilon BJ+\epsilon J B+ib+aj).
\end{align*}
and abuse the notation to denote $T(\mbf{v})$ as $ker(\tau(\mbf{v}))/im(\sigma({\mbf{v}}))$. Then $T(\mbf{v})$ is the tangent bundle of $\mf{M}(\mbf{v},\mbf{w})$. Moreover, the restriction of the perfect pairing $\Omega(\mbf{v},\mbf{v})$ to the diagonal is a symplectic form of $T(\mbf{v})$, which we denote as $\Omega(\mbf{v})$.
\end{theorem}

\begin{corollary}
  \label{011515}
  The coherent sheaf $\mc{O}_{\Delta \mf{M}(\mbf{v})}$ is resolved by the Koszul complex of $s^{2}(\mbf{v},\mbf{v})$.
\end{corollary}

\subsection{The case that $\mbf{v}^{2}=\mbf{v}^{1}+\delta_{k}$}
Now we consider the case that  $\mathbf{v}^{2}=\mathbf{v}^{1}+\delta_{k}$ for some $k\in I$. As $\mf{P}(\mbf{v}^{2},\mbf{v}^{1})=\emptyset$, $s^{2}(\mbf{v}^{1},\mbf{v}^{2})$ is surjective and we denote $\mc{N}(\mbf{v}^{1},\mbf{v}^{2})$ as its kernel. The image of $s^{1}(\mbf{v}^{2},\mbf{v}^{2})$ is in $\mc{N}(\mbf{v}^{1},\mbf{v}^{2})$ and hence it induces a global section 
$$n(\mbf{v}^{1},\mbf{v}^{2}):\mc{O}\to \mc{N}(\mbf{v}^{1},\mbf{v}^{2}).$$
Then $\mf{P}(\mbf{v}^{1},\mbf{v}^{2})$ is the zero locus of $n(\mbf{v}^{1},\mbf{v}^{2})$. For any closed point $[\mc{B}^{1}]\subset [\mc{B}^{2}]\in \mf{P}(\mbf{v}^{1},\mbf{v}^{2})$, we have
$$\mc{N}(\mbf{v}^{1},\mbf{v}^{2})|_{[\mc{B}^{1}]\subset [\mc{B}^{2}]}\cong Ext^{1}(\mc{B}^{1},\mc{B}^{2}).$$

\begin{theorem}
  \label{nest}    
  The variety $\mf{P}(\mbf{v}^{1},\mbf{v}^{2})$ is smooth, and the structure sheaf of $\mf{P}(\mbf{v}^{1},\mbf{v}^{2})$ in $\mf{M}(\mbf{v}^{1},\mbf{w})\times \mf{M}(\mbf{v}^{2},\mbf{w})$ is resolved by the Koszul complex of $n^{\vee}(\mbf{v}^{1},\mbf{v}^{2})$.
\end{theorem}
\begin{proof}
  We just need to prove the restriction morphism
  $$T(\mbf{v}^{1})\oplus T(\mbf{v}^{2})|_{[\mc{B}^{1}]\subset [\mc{B}^{2}]}\to \mc{N}(\mbf{v}^{1},\mbf{v}^{2})|_{[\mc{B}^{1}]\subset [\mc{B}^{2}]}$$
  is surjective. By \cref{thm35}
  $$T(\mbf{v}^{1})\oplus T(\mbf{v}^{2})|_{[\mc{B}^{1}]\subset [\mc{B}^{2}]}|_{[\mc{B}^{1}]\subset [\mc{B}^{2}]}\cong Ext^{1}(\mc{B}^{1},\mc{B}^{1})\oplus Ext^{1}(\mc{B}^{2},\mc{B}^{2})$$
  and the morphism
  $$Ext^{1}(\mc{B}^{1},\mc{B}^{1})\oplus Ext^{1}(\mc{B}^{2},\mc{B}^{2})\xrightarrow{(-\otimes \zeta_{\mc{B}^{1}\mc{B}^{2}})+(\zeta_{\mc{B}^{1}\mc{B}^{2}}\otimes -)}Ext^{1}(\mc{B}^{1},\mc{B}^{2})$$ is surjective by \cref{011209}.
\end{proof}

We consider the morphism
\begin{gather*}
  t(\mbf{v}^{1},\mbf{v}^{2}):\mf{P}(\mbf{v}^{1},\mbf{v}^{2})\to \mbf{L}_{k}:  \\
  t(\mbf{v}^{1},\mbf{v}^{2})([B^{1},i^{1},j^{1}],[B^{2},i^{2},j^{2}]):=(Tr(B_{l}^{1})-Tr(B_{l}^{2}))_{l\in E_{k}^{\#}}.
\end{gather*}
For a pair $[\mc{B}^{1}]\subset [\mc{B}^{2}]$, let $x=t(\mbf{v}^{1},\mbf{v}^{2})$. Then we have $[\mc{B}^{1}]\subset_{x} [\mc{B}^{2}]$. We notice the following morphisms of tangent spaces and cotangent spaces
\begin{align*}
  dt(\mbf{v}^{1},\mbf{v}^{2}):T_{\mf{P}(\mbf{v}^{1},\mbf{v}^{2})(\mbf{v})}\to T\mbf{L}_{k},\quad  dt^{*}(\mbf{v}^{1},\mbf{v}^{2}):T\mbf{L}_{k}\to T_{\mf{P}(\mbf{v}^{1},\mbf{v}^{2})(\mbf{v})}^{*}
\end{align*}
satisfy
$$dt(\mbf{v}^{1},\mbf{v}^{2})|_{[\mc{B}^{1}]\subset_{x}[\mc{B}^{2}]}=dt_{\mc{B}^{1}\mc{B}^{2}}, \quad dt^{*}(\mbf{v}^{1},\mbf{v}^{2})|_{[\mc{B}^{1}]\subset_{x}[\mc{B}^{2}]}=dt^{*}_{\mc{B}^{1}\mc{B}^{2}}.$$
\begin{theorem}
   \label{020527}\label{0209310}
Under the closed embedding $p\times q\times t(\mbf{v}^{1},\mbf{v}^{2})$, $\mf{P}(\mbf{v}^{1},\mbf{v}^{2})$ is a Lagrangian of $\mf{M}(\mbf{v}^{1},\mbf{w})\times \mf{M}(\mbf{v}^{2},\mbf{w})\times \mbf{L}_{k}$. Thus
$$dim(\mf{P}(\mbf{v}^{1},\mbf{v}^{2}))=\frac{1}{2}(dim (\mf{M}(\mbf{v}^{1},\mbf{w}))+ dim (\mf{M}(\mbf{v}^{1},\mbf{w}))+ 2g_{k}).$$
\end{theorem}
\begin{proof}
  It follows from \cref{e0204111}.
\end{proof}

\subsection{Universal complexes on $\mf{M}(\mbf{v},k)\times \mbf{L}_k$}

Given $k\in I$, we abuse the notation to denote $\mb{C}_{k}$ as the trivial line bundle on $\mbf{L}_{k}$. Then we consider the following complexes of locally free sheaves on $\mf{M}(\mbf{v},\mbf{w})\times \mbf{L}_{k}$:

\begin{align*}
  \mc{E}xt^{\bullet}(k,\mathbf{v}):=I(\mb{C}_{k},V)\xrightarrow{\sigma(k,\mbf{v})}E(\mb{C}_{k},V)\oplus I(\mb{C}_{k},W)\xrightarrow{\tau(k,\mbf{v})}I(\mb{C}_{k},V), \\
  \mc{E}xt^{\bullet}(\mbf{v},k):= I(V,\mb{C}_{k})\xrightarrow{\sigma(\mbf{v},k)}E(V,\mb{C}_{k})\oplus I(W,\mb{C}_{k})\xrightarrow{\tau(\mbf{v},k)}I(V_{k},\mb{C}_{k}).
\end{align*}
where
\begin{align*}
  \sigma(k,\mbf{v})(a):=(B-x)a\oplus j_{k}a, \quad  \tau(k,\mbf{v})(C\oplus D):=\epsilon (B-x)C+i_{k}D, \\
  \tau(\mbf{v},k)(a):=a\epsilon (B-x)\oplus ai_{k},\quad \sigma(\mbf{v},k)(C\oplus D):=C(B-x)+Cj_{k}.
\end{align*}
Then for $([\mc{B}],x)\in \mf{M}(\mbf{v},\mbf{w})\times \mbf{L}_{k}$,  we have
$$\mc{E}xt^{\bullet}(\mbf{v},k)|_{([\mc{B}],x)}\cong \mc{E}xt^{\bullet}(\mc{B},x), \quad \mc{E}xt^{\bullet}(k,\mbf{v})|_{([\mc{B}],x)}\cong \mc{E}xt^{\bullet}(x,\mc{B}).$$
The morphism $\sigma(k,\mbf{v})$ and $\tau(\mbf{v},k)$ are dual to each other, and $\sigma(\mbf{v},k)$ and $\tau(k,\mbf{v})$ are dual to each other. By \cref{112302}, $\sigma(k,\mbf{v})$ is injective at each closed point. We denote $U_{k}$ as the cokernel of  $\sigma(k,\mbf{v})$ which is a locally free sheaf on $\mf{M}(\mbf{v},\mbf{w})\times \mbf{L}_{k}$. We denote
$$\mc{I}_{k}(\mbf{v}):U_{k}\xrightarrow{u_{k}} V_{k}\quad \mc{I}_{k}^{\vee}(\mbf{v}):V_{k}^{\vee}\xrightarrow{u_{k}^{\vee}}U_{k}^{\vee}, $$
where $u_{k}$ is the canonical morphism induced from $\tau(k,\mbf{v})$.

\subsection{Tautological line bundles on the nested quiver variety}
\label{s013123}
On $\mf{P}(\mbf{v}^{1},\mbf{v}^{2})$, the cokernel of $\zeta_{\mbf{v}^{1},\mbf{v}^{2}}:V^{1}\to V^{2}$ is a line bundle, which we denote as $\mc{L}$. At each closed point $[\mc{B}^{1}]\subset [\mc{B}^{2}]$, the fiber of $\mc{L}$ is $V^{2}/\zeta_{\mc{B}^{1}\mc{B}^{2}} V^{1}$.

On $\mf{P}(\mbf{v}^{1},\mbf{v}^{2})$, we denote $U_{k}^{1}$ (resp. $U_{k}^{2}$)as the pull back of $U_{k}$ from $\mf{M}(\mbf{v}^{1},\mbf{w})\times \mbf{L}_{k}$ (resp. $\mf{M}(\mbf{v}^{2},\mbf{w})\times \mbf{L}_{k}$).  The morphism $\zeta_{\mbf{v}^{1},\mbf{v}^{2}}$ induces the diagram locally free sheaves:
\begin{equation}
  \begin{tikzcd}
    0 \ar{d} & 0\ar{d} & 0 \ar{d} \\
    I(\mb{C}_{k},V^{1})\ar{r}{\sigma(k,\mbf{v}^{1})} \ar{d} & E(\mb{C}_{k},V^{1})\oplus I(\mb{C}_{k},W)\ar{r}{\tau(k,\mbf{v}^{1})}\ar{d} & I(\mb{C}_{k},V^{1})\ar{d} \\
    I(\mb{C}_{k},V^{2})\ar{r}{\sigma(k,\mbf{v}^{2})}\ar{d} & E(\mb{C}_{k},V^{2})\oplus I(\mb{C}_{k},W)\ar{r}{\tau(k,\mbf{v}^{2})}\ar{d} & I(\mb{C}_{k},V^{2})\ar{d} \\
    \mc{L}\ar{r}{0}\ar{d} & \mc{L}^{\oplus 2g_{k}}\ar{r}{0}\ar{d} &\mc{L}\ar{d} \\
    0 & 0& 0
  \end{tikzcd}
\end{equation}
where all the columns are short exact sequences. Hence it induces the following commutative diagram of complexes:
\begin{equation}
  \label{e012623}
  \begin{tikzcd}
    \mc{L}\ar{r}{\gamma_{\mbf{v}^{1},k}^{+\vee}} & U_{k}^{1}\ar{r}{u_{k}^{1}}\ar{d} & V_{k}^{1}\ar{d} \\
    & U_{k}^{2}\ar{r}\ar{r}{u_{k}^{2}} & V_{k}^{2}\ar{r}{\gamma_{\mbf{v}^{2},k}^{-}} & \mc{L}
  \end{tikzcd}
\end{equation}
such that $\gamma_{{\mbf{v}^{1},k}}^{+}:=(\gamma_{{\mbf{v}^{1},k}}^{+\vee})^{\vee}$ and $\gamma_{\mbf{v}^{2},k}^{-}$ are both surjective.

Finally, we realize the nested quiver variety through Grothendieck's projectivization. Let $X$ be a scheme and $F$ be a coherent sheaf on $X$.

\begin{definition}
  \label{def:4.11}
The Grothendieck's projectivization of $F$ over $X$, which we denote as $P_{X}(F)$, is the $X$-scheme, such that for any $X$-scheme $f:T\to X$, $Hom_{X}(T,P_{X}(F))$ is parametrized by the surjective morphisms:
$$f^{*}F\to L$$
where $L$ is a line bundle on $X$. Given a two term complex $f:\mc{F}\to \mc{G}$ of locally free sheaves over $X$, we define $P_{X}(f):=P_{X}(coker(f))$ and let $\mc{O}_{P_{X}(f)}(1)$ be the universal line bundle.  
\end{definition}

\begin{proposition}
  \label{0212310}
  The morphism  $\gamma_{{\mbf{v}^{1},k}}^{+}$ induces a canonical isomorphism
  $$\mf{P}(\mbf{v}^{1},\mbf{v}^{2})\cong P_{\mf{M}(\mbf{v}^{1},\mbf{w})\times \mbf{L}_{k}}(\mc{I}_{k}(\mbf{v}^{1})^{\vee})$$
  where $p\times t(\mbf{v}^{1},\mbf{v}^{2})$ is the canonical projection morphism under the isomorphism.

  The morphism  $\gamma_{{\mbf{v}^{2},k}}^{-}$ induces a canonical isomorphism
  $$\mf{P}(\mbf{v}^{1},\mbf{v}^{2})\cong P_{\mf{M}(\mbf{v}^{2},\mbf{w})\times \mbf{L}_{k}}(\mc{I}_{k}(\mbf{v}^{2}))$$
  where $q\times t(\mbf{v}^{1},\mbf{v}^{2})$ is the canonical projection morphism under the isomorphism. Moreover, under the above isomorphisms, we have
  $$\mc{O}_{P_{\mf{M}(\mbf{v}^{1},\mbf{w})\times \mbf{L}_{k}}(\mc{I}_{k}(\mbf{v}^{1})^{\vee})}(-1)\cong \mc{L}\cong \mc{O}_{P_{\mf{M}(\mbf{v}^{2},\mbf{w})\times \mbf{L}_{k}}(\mc{I}_{k}(\mbf{v}^{2}))}(1).$$
\end{proposition}
\begin{proof}
  It follows from \cref{012727} and \cref{012728}.
\end{proof}
\section{The Triple and Quadruple moduli space}
\label{sec5}

In this section, we introduce several triple and quadruple nested quiver varieties. we fix $k\in I$ and choose dimension vectors $\mbf{v}^{0}+\delta_{k}=\mbf{v}^{1}=\mbf{v}^{1'}=\mbf{v}^{2}-\delta_{k}$.

\subsection{The triple nested quiver variety $\mf{J}_{k}(\mbf{v}^{1})$}

\begin{definition}
   We define the nested triple variety $$\mf{J}_{k}(\mbf{v}^{1}):=\mf{P}(\mbf{v}^{0},\mbf{v}^{1})\times_{\mf{M}(\mbf{v},\mbf{w})\times \mbf{L}_{k}}\mf{P}(\mbf{v}^{1},\mbf{v}^{2})$$
which consists of triple
$$[\mc{B}^{0}]\subset_{x}[\mc{B}^{1}]\subset_{x}[\mc{B}^{2}]$$
such that $x\in \mbf{L}_{k}$ and $[\mc{B}^{l}]\in \mf{M}(\mbf{v}^{l},\mbf{w})$, $l=1,2,3$.
\end{definition}
Over $\mf{J}_{k}(\mbf{v})$, there are two canonical $I$-graded injective morphisms of locally free sheaves $\zeta^{1}:V^{0}\to V^{1},\zeta^{2}:V^{1}\to V^{2}$, and we denote $\mc{L}^{l}$ as the cokernel of $\zeta^{l}$ for $l=1,2$, and both $\mc{L}^{1}$ and $\mc{L}^{2}$ are line bundles on $\mf{J}_{k}(\mbf{v})$.

\begin{proposition}
  \label{020531}
  The triple nested quiver variety $\mf{J}_{k}(\mbf{v}^{1})$ is smooth and its dimension is $dim(\mf{M}(\mbf{v}^{1},\mbf{w}))+2g_{k}-1$. 
\end{proposition}
\begin{proof}
  Let $[\mc{B}^{0}]\subset_{x}[\mc{B}^{1}]\subset_{x}[\mc{B}^{2}]$, by  \cref{0113011}, we have the long exact sequences:
  \begin{align*}
    0\to \mb{C}\xrightarrow{\theta(\zeta_{\mc{B}^{1}x}\otimes \zeta_{x\mc{B}^{1}})\oplus \theta(\zeta_{\mc{B}^{1}x}\otimes \zeta_{x\mc{B}^{1}})\ }Ext^{1}(\mc{B}^{1},\mc{B}^{1})\oplus Ext^{1}(x,x)\\
    \xrightarrow{} T_{\mc{B}^{0}\mc{B}^{1}}^{*}\oplus T_{\mc{B}^{1}\mc{B}^{2}}^{*}\to T^{*}_{\mf{J}_{k}(\mbf{v}^{1})}|_{[\mc{B}^{0}]\subset_{x}[\mc{B}^{1}]\subset_{x}[\mc{B}^{2}]}.
  \end{align*}
  and hence $dim(T^{*}_{\mf{J}_{k}(\mbf{v}^{1})}|_{[\mc{B}^{0}]\subset_{x}[\mc{B}^{1}]\subset_{x}[\mc{B}^{2}]})$ is
  \begin{align*}
  dim(\mf{P}(\mbf{v}^{0},\mbf{v}^{1}))+dim(\mf{P}(\mbf{v}^{1},\mbf{v}^{2}))-dim(\mf{M}(\mbf{v}^{1},\mbf{w}))-2g_{k}+1\\=dim(\mf{M}(\mbf{v}^{1},\mbf{w}))+2g_{k}-1.   
  \end{align*}
\end{proof}

\subsection{The decreasing/increasing triple quiver variety $\mf{Z}_{k}^{\pm}(\mbf{v}^{1})$ and Negu\c{t}'s quadruple moduli space $\mf{Y}_{k}(\mbf{v}^{1})$}

\begin{definition}
   We consider the decreasing tripled quiver variety
$$\mf{Z}_{k}^{-}(\mbf{v}^{1}):=\mf{P}(\mbf{v}^{0},\mbf{v}^{1})\times_{\mbf{M}(\mbf{v}^{0},\mbf{w})}\mf{P}(\mbf{v}^{0},\mbf{v}^{1'})$$
which consists of decreasing triples $([\mc{B}^{0}],[\mc{B}^{1}],[\mc{B}^{1'}])\in \mf{M}(\mbf{v}^{0},\mbf{w})\times \mf{M}(\mbf{v}^{1},\mbf{w})\times \mf{M}(\mbf{v}^{1'},\mbf{w})$
such that $[\mc{B}^{0}]\subset_{x}[\mc{B}^{1}]$ and $[\mc{B}^{0}]\subset_{y} [\mc{B}^{1'}]$ for some $x,y\in \mbf{L}_{k}$.
\end{definition}

\begin{definition}
We define the increasing tripled quiver variety
$$\mf{Z}_{k}^{+}(\mbf{v}^{1}):=\mf{P}(\mbf{v}^{1},\mbf{v}^{2})\times_{\mbf{M}(\mbf{v}^{2},\mbf{w})}\mf{P}(\mbf{v}^{1'},\mbf{v}^{2})$$
which consists of decreasing triples  $([\mc{B}^{2}],[\mc{B}^{1}],[\mc{B}^{1'}])\in \mf{M}(\mbf{v}^{2},\mbf{w})\times \mf{M}(\mbf{v}^{1},\mbf{w})\times \mf{M}(\mbf{v}^{1'},\mbf{w})$
such that $[\mc{B}^{2}]_{y}\supset [\mc{B}^{1}]$ and $[\mc{B}^{2}]_{x}\supset [\mc{B}^{1'}]$ for some $x,y\in \mbf{L}_{k}$.
\end{definition}

\begin{definition}[Negu\c{t}'s quadruple moduli space]
  \label{021354}
  We define the Negu\c{t}'s quadruple moduli space $\mf{Y}_{k}(\mbf{v}^{1})$ which consists of quadruples:
  $$([\mc{B}^{0}],[\mc{B}^{1}],[\mc{B}^{1'}],[\mc{B}^{2}])\in \mf{M}(\mbf{v}^{0},\mbf{w})\times \mf{M}(\mbf{v}^{1},\mbf{w})\times \mf{M}(\mbf{v}^{1'},\mbf{w})\times \mf{M}(\mbf{v}^{2},\mbf{w})$$
  such that
  $$([\mc{B}^{0}]\subset_{x} [\mc{B}^{1}]),([\mc{B}^{0}]\subset_{y}[\mc{B}^{1'}]),([\mc{B}^{1}]\subset_{y}[\mc{B}^{2}]),([\mc{B}^{1'}]\subset_{x}[\mc{B}^{2}])$$
  for some $x,y\in \mbf{L}_{k}$.
\end{definition}

We consider the diagonal morphisms
\begin{align*}
  \Delta_{p(\mbf{v}^{0},\mbf{v}^{1})}:\mf{P}(\mbf{v}^{0},\mbf{v}^{1})\to \mf{Z}_{k}^{-}(\mbf{v}^{1}), \quad  ([\mc{B}^{0}]\subset [\mc{B}^{1}])\to ([\mc{B}^{0}],[\mc{B}^{1}],[\mc{B}^{1}]), \\
  \Delta_{q(\mbf{v}^{1},\mbf{v}^{2})}:\mf{P}(\mbf{v}^{1},\mbf{v}^{2})\to \mf{Z}_{k}^{+}(\mbf{v}^{1}) \quad  ([\mc{B}^{2}]\supset [\mc{B}^{1}])\to ([\mc{B}^{2}],[\mc{B}^{1}],[\mc{B}^{1}]).
\end{align*}
and the close embedding 
\begin{align*}
  \Delta_{\mf{J}_{k}(\mbf{v}^{1})}:\mf{J}_{k}(\mbf{v}^{1})&\to \mf{Y}_{k}(\mbf{v}^{1}) \\
  ([\mc{B}^{0}],[\mc{B}^{1}],[\mc{B}^{2}])&\to ([\mc{B}^{0}],[\mc{B}^{1}],[\mc{B}^{1}],[\mc{B}^{2}])
\end{align*}
By forgetting $[\mc{B}^{0}]$ and $[\mc{B}^{2}]$ respectively, we have morphisms
$$\alpha^{+}:\mf{Y}_{k}(\mbf{v}^{1})\to \mf{Z}^{+}_{k}(\mbf{v}^{1}),\quad \alpha^{-}:\mf{Y}_{k}(\mbf{v}^{1})\to \mf{Z}^{-}_{k}(\mbf{v}^{1}).$$

\begin{proposition}
  \label{012531}
  The morphism $\alpha^{+}$ and $\alpha^{-}$ induces canonical isomorphisms
  \begin{equation}
    \label{eq:neqqq}
    \mf{Z}_{k}^{+}(\mbf{v}^{1})-\mf{P}(\mbf{v}^{1},\mbf{v}^{2})\cong \mf{Y}_{k}(\mbf{v}^{1})-\mf{J}_{k}(\mbf{v}^{1})\cong \mf{Z}_{k}^{-}(\mbf{v}^{1})-\mf{P}(\mbf{v}^{0},\mbf{v}^{1})
  \end{equation}
  and they are both smooth varieties of $dim(\mf{M}(\mbf{v}^{1},\mbf{w}))+2g_{k}$ dimension.
\end{proposition}
\begin{proof}
  Let $[\mc{B}^{2}]_{y}\supset [\mc{B}^{1}]$ and $[\mc{B}^{2}]_{x}\supset [\mc{B}^{1'}]$ for some $x,y\in \mbf{L}_{k}$ such that $[\mc{B}^{1}]\neq [\mc{B}^{1'}]$. Then $\mc{B}^{1}\not\cong \mc{B}^{1'}$ and $\mc{B}^{1}\oplus \mc{B}^{1'}\xrightarrow{\zeta_{\mc{B}^{1}\mc{B}^{2}}+ \zeta_{\mc{B}^{1'}\mc{B}^{2}}}\mc{B}^{2}$ is surjective. Let $\mc{B}^{0}$ be its kernel, and it is easy to show that $[\mc{B}^{0}]\in \mf{M}(\mbf{v}^{0},\mbf{w})$.
  
  On the other hand, $[\mc{B}^{0}]\subset_{x}[\mc{B}^{1}]$ and $[\mc{B}^{0}]\subset_{y} [\mc{B}^{1'}]$ for some $x,y\in \mbf{L}_{k}$ such that $[\mc{B}^{1}]\neq [\mc{B}^{1'}]$. Then $\mc{B}^{0}\xrightarrow{\zeta_{\mc{B}^{0}\mc{B}^{1}}\oplus \zeta_{\mc{B}^{0}\mc{B}^{1'}}}\mc{B}^{1}\oplus \mc{B}^{1}$ is injective. Let $\mc{B}^{2}$ be its cokernel, and it is also easy to show that $[\mc{B}^{2}]\in \mf{M}(\mbf{v}^{2},\mbf{w})$. Thus we prove \cref{eq:neqqq}.
  
  Finally, the smoothness of $\mf{Z}_{k}^{+}(\mbf{v}^{1})-\mf{P}(\mbf{v}^{1},\mbf{v}^{2})$ follows from \cref{0124110}. Moreover, its dimension is
  \begin{align*}
    2\times \frac{1}{2}(dim(\mf{M}(\mbf{v}^{2},\mbf{w}))+dim(\mf{M}(\mbf{v}^{1},\mbf{w}))+g_{k})-dim(\mf{M}(\mbf{v}^{1},\mbf{w}))\\=dim(\mf{M}(\mbf{v}^{1},\mbf{w}))+2g_{k}.
  \end{align*}
\end{proof}

\subsection{The smoothness of $\mf{Y}_{k}(\mbf{v}^{1})$} Finally, we prove the smoothness of the Negu\c{t}'s quadruple moduli space:
\begin{theorem}
  \label{021147}
  The quadruple moduli space $\mf{Y}_{k}(\mbf{v}^{1})$ is smooth and $\mf{J}_{k}{(\mbf{v}^{1})}$ is a regular divisor of $\mf{Y}_{k}(\mbf{v}^{1})$.
\end{theorem}

\begin{proof}
  Given quadruples
  $$([\mc{B}^{0}]\subset_{x} [\mc{B}^{1}]),([\mc{B}^{0}]\subset_{y}[\mc{B}^{1'}]),([\mc{B}^{1}]\subset_{y}[\mc{B}^{2}]),([\mc{B}^{1'}]\subset_{x}[\mc{B}^{2}]),$$
  the tangent space of $\mf{Y}_{k}(\mbf{v}^{1})$ at this point is the subspace of $(b^{0},b^{1},b^{1'},b^{2})\in$
  $$ Ext^{1}(\mc{B}^{0},\mc{B}^{0})\oplus Ext^{1}(\mc{B}^{1},\mc{B}^{1})\oplus Ext^{1}(\mc{B}^{1'},\mc{B}^{1'})\oplus Ext^{1}(\mc{B}^{2},\mc{B}^{2})$$
  such that
  \begin{enumerate}
  \item $(b^{0},b^{1})\in T_{\mc{B}^{0}\mc{B}^{1}}$, $(b^{0},b^{1'})\in T_{\mc{B}^{0}\mc{B}^{1'}}$, $(b^{1},b^{2})\in T_{\mc{B}^{1}\mc{B}^{2}}$ and $(b^{1'},b^{2})\in T_{\mc{B}^{1'}\mc{B}^{2}}$,
  \item $dt_{\mc{B}^{0}\mc{B}^{1}}(b^{0},b^{1})=dt_{\mc{B}^{1'}\mc{B}^{2}}(b^{1'},b^{2})$, 
  \item $dt_{\mc{B}^{0}\mc{B}^{1'}}(b^{0},b^{1'})=dt_{\mc{B}^{1}\mc{B}^{2}}(b^{1},b^{2})$.
  \end{enumerate}
  By \cref{012531}, $\mf{Y}_{k}(\mbf{v}^{1})-\mf{J}_{k}(\mbf{v}^{1})$ is smooth and its dimension is $dim(\mf{J}_{k}(\mbf{v}^{1}))+1$. We consider a closed point $([\mc{B}^{0}],[\mc{B}^{1}],[\mc{B}^{2}])\in \mf{J}_{k}(\mbf{v}^{1})$. For a quadruple $(b^{0},b^{1},b^{1'},b^{2})$ in the tangent space of $\mf{Y}_{k}(\mbf{v}^{1})$, then $(b^{0},\frac{1}{2}(b^{1}+b^{1'}),\frac{1}{2}(b^{1}+b^{1'}),b^{2})$ and $(0,\frac{1}{2}(b^{1}-b^{1'}),\frac{1}{2}(b^{1'}-b^{1}),0)$ are also in the tangent space of $\mf{Y}_{k}(\mbf{v}^{1})$. We notice that for $b\in Ext^{1}(\mc{B}^{1},\mc{B}^{1})$, $dt_{\mc{B}^{0}\mc{B}^{1}}(0,b)=dt_{\mc{B}^{1}\mc{B}^{2}}(b,0)$. Hence the tangent space $T_{\mf{Y}_{k}(\mbf{v}^{1})}$ at $([\mc{B}^{0}],[\mc{B}^{1}],[\mc{B}^{2}])$ decompose into two subspaces: one is $T_{\mf{J}_{k}(\mbf{v}^{1})}$, and another is
  $$N:=\{b\in Ext^{1}(\mc{B}^{1},\mc{B}^{1})|(b,0)\in T_{\mc{B}^{1}\mc{B}^{2}}, (0,b)\in T_{\mc{B}^{0}\mc{B}^{1}}\}$$
 By \cref{0113011}, $N$ is one dimensional and thus $dim(T_{\mf{Y}_{k}(\mbf{v}^{1})})|_{([\mc{B}^{0}],[\mc{B}^{1}],[\mc{B}^{2}])}=dim (\mf{Y}_{k}(\mbf{v}^{1}))$. Hence $\mf{Y}_{k}(\mbf{v}^{1})$ is smooth and $\mf{J}_{k}(\mbf{v}^{1})$ is a regular divisor.
\end{proof}

\section{The Derived Blow-up of the Diagonal}
\label{sec6}
Let $\Phi:Y\to X$ be a morphism of smooth varieties. Let $\mb{R}\Delta_{\Phi}:Y\to Y\times_{X}^{\mb{L}}Y$ be the diagonal morphism, which is a closed embedding and $N_{\mb{R}\Delta_{\Phi}}=L_{\Phi}$. By the gluing theorem \cref{020656}, we have
\begin{corollary}
  \label{cor7}
  If $Y\times_{X}^{\mb{L}}Y-Y$ is smooth and $\mb{P}_{Y}(L_{\Phi})$ is smooth, then $Bl_{\Delta_{\Phi}}\cong \mb{B}l_{\Delta_{\Phi}}$ which are both smooth.
\end{corollary}

Now we apply the derived blow-up theory to quiver varieties. Given $k\in I$, and we choose dimension vectors $\mbf{v}^{0}+\delta_{k}=\mbf{v}^{1}=\mbf{v}^{1'}=\mbf{v}^{2}-\delta_{k}$.
We consider the following derived enhancement of 
$$\mb{R}\mf{Z}_{k}^{-}(\mbf{v}^{1}):=\mf{P}(\mbf{v}^{1},\mbf{v}^{2})\times^{\mb{L}}_{\mbf{M}(\mbf{v}^{0},\mbf{w})}\mf{P}(\mbf{v}^{1'},\mbf{v}^{2}),\quad \mb{R}\mf{Z}_{k}^{+}(\mbf{v}^{1}):=\mf{P}(\mbf{v}^{1},\mbf{v}^{2})\times_{\mbf{M}(\mbf{v}^{2},\mbf{w})}^{\mb{L}}\mf{P}(\mbf{v}^{1'},\mbf{v}^{2})$$
together with the diagonal morphism 
\begin{align*}
  \mb{R}\Delta_{p(\mbf{v}^{0},\mbf{v}^{1})}:\mf{P}(\mbf{v}^{0},\mbf{v}^{1})\to \mb{R}\mf{Z}_{k}^{-}(\mbf{v}^{1}),\quad \mb{R}\Delta_{q(\mbf{v}^{1},\mbf{v}^{2})}:\mf{P}(\mbf{v}^{1},\mbf{v}^{2})\to \mb{R}\mf{Z}_{k}^{+}(\mbf{v}^{1}).
\end{align*}

In this section, we are going to prove 
\begin{theorem}[\cref{0206412}]
\label{newthm}
  We have
  \begin{align*}
    \mf{J}_{k}(\mbf{v}^{1}) \cong \mb{P}_{\mf{P}(\mbf{v}^{0},\mbf{v}^{1})}(L_{p(\mbf{v}^{0},\mbf{v}^{1})})\cong  \mb{P}_{\mf{P}(\mbf{v}^{1},\mbf{v}^{2})}(L_{q(\mbf{v}^{1},\mbf{v}^{2})}).
  \end{align*}
  as smooth schemes.
\end{theorem}

By \cref{newthm}, \cref{012531} and \cref{cor7}, we have
\begin{theorem}
  \label{020658}
  Both $\mb{B}l_{\mb{R}\Delta_{p(\mbf{v}^{0},\mbf{v}^{1})}}$ and $\mb{B}l_{\mb{R}\Delta_{q(\mbf{v}^{1},\mbf{v}^{2})}}$ are smooth schemes, and we have $\mb{B}l_{\mb{R}\Delta_{p(\mbf{v}^{0},\mbf{v}^{1})}}\cong Bl_{\Delta_{p(\mbf{v}^{0},\mbf{v}^{1})}}$ and $\mb{B}l_{\mb{R}\Delta_{q(\mbf{v}^{1},\mbf{v}^{2})}}\cong Bl_{\Delta_{q(\mbf{v}^{1},\mbf{v}^{2})}}$.
\end{theorem}

Finally, we consider the morphism $\alpha_{k}^{\pm}(\mbf{v}^{1})$. We notice that the preimage of the diagonal $\mf{P}(\mbf{v}^{0},\mbf{v}^{1})$ (resp. $\mf{P}(\mbf{v}^{1},\mbf{v}^{2})$) under the morphism $\alpha_{k}^{-}(\mbf{v}^{1})$ (resp. $\alpha_{k}^{+}(\mbf{v}^{1})$) is $\mf{J}_{k}(\mbf{v}^{1})$, which is a regular divisor of $\mf{Y}_{k}(\mbf{v}^{1})$. Hence morphisms $\alpha_{k}^{\pm}(\mbf{v}^{1})$ factors through morphisms

$$\overline{\alpha_{k}^{-}(\mbf{v}^{1})}:\mf{Y}_{k}(\mbf{v}^{1})\to Bl_{\Delta_{p(\mbf{v}^{0},\mbf{v}^{1})}}, \quad\alpha_{k}^{+}(\mbf{v}^{1}):\overline{\alpha_{k}^{+}(\mbf{v}^{1})}: \mf{Y}_{k}(\mbf{v}^{1})\to Bl_{\Delta_{q_{k}(\mbf{v}^{2})}}.$$

\begin{theorem}
  \label{0213711}
    Both $\overline{\alpha_{k}^{\pm}(\mbf{v}^{1})}$ are isomorphisms. Moreover, if $dim(\mf{P}(\mbf{v}^{0},\mbf{v}^{1}))>dim(\mf{M}(\mbf{v}^{0},\mbf{w}))$ (resp. $dim(\mf{P}(\mbf{v}^{1},\mbf{v}^{2}))>dim(\mf{M}(\mbf{v}^{2},\mbf{w}))$), then $\mf{Z}_{k}^{-}(\mbf{v}^{1})$ (resp. $\mf{Z}_{k}^{+}(\mbf{v}^{1})$) is a canonical singularity. 
  \end{theorem}
    \begin{proof}
      By \cref{020658}, $\overline{\alpha_{k}^{\pm}(\mbf{v}^{1})}$ are both isomorphism when restricting to $\mf{J}_{k}(\mbf{v}^{1})$ and $\mf{Y}_{k}(\mbf{v}^{1})-\mf{J}_{k}(\mbf{v}^{1})$. Hence $\overline{\alpha_{k}^{\pm}(\mbf{v}^{1})}$ are both etale, birational, dominant, and hence isomorphisms.
      The singularity property follows from the gluing theorem \cref{020656}.
  \end{proof}

\subsection{Projectivization and nested quiver varieties}
The following lemma is a derived enhancement of \cref{0212310}:
\begin{lemma}
  \label{020147}
  We have
  $$ \mb{P}_{\mf{M}(\mbf{v}^{1},\mbf{w})\times \mbf{L}_{k}}(\mc{I}_{k}^{\vee}(\mbf{v}^{1})) \cong \mf{P}(\mbf{v}^{1},\mbf{v}^{2})\cong \mb{P}_{\mf{M}(\mbf{v}^{2},\mbf{w})\times \mbf{L}_{k}}(\mc{I}_{k}(\mbf{v}^{2}))$$
   Moreover, under the above isomorphism, we have
  $$\mc{L}\cong \mc{O}_{\mb{P}_{\mf{M}(\mbf{v}^{1},\mbf{w})\times \mbf{L}_{k}}(\mc{I}_{k}^{\vee}(\mbf{v}^{1}))}(-1)\cong \mc{O}_{\mb{P}_{\mf{M}(\mbf{v}^{2},\mbf{w})\times \mbf{L}_{k}}(\mc{I}_{k}(\mbf{v}^{2}))}(1).$$
\end{lemma}
\begin{proof}
  By \cref{0212310}, we only need to prove that
  $$vdim(\mb{P}_{\mf{M}(\mbf{v}^{1},\mbf{w})\times \mbf{L}_{k}}(\mc{I}_{k}^{\vee}(\mbf{v}^{1}))) \cong dim(\mf{P}(\mbf{v}^{1},\mbf{v}^{2}))\cong vdim(\mb{P}_{\mf{M}(\mbf{v}^{2},\mbf{w})\times \mbf{L}_{k}}(\mc{I}_{k}(\mbf{v}^{2})))$$
     We notice that
  \begin{align*}
    rank(\mc{I}_{k}(\mbf{v}^{1}))&=<\mbf{v}^{1},\delta_{k}>_{Q}-\mbf{w}\bullet \delta_{k} \\
                                 &= \frac{1}{2}(<\mbf{v}^{2},\mbf{v}^{2}>_{Q}-<\mbf{v}^{1},\mbf{v}^{1}>_{Q}-<\delta_{k},\delta_{k}>-2\mbf{w}\bullet \mbf{v}^{2}+2\mbf{w}\bullet \mbf{v}^{1})\\
                                 &= \frac{1}{2}(-dim(\mf{M}(\mbf{v}^{2},\mbf{w}))+dim(\mf{M}(\mbf{v}^{1},\mbf{w}))+2g_{k}-2)
  \end{align*}
  and
  \begin{align*}
    rank(\mc{I}_{k}(\mbf{v}^{2}))= <\mbf{v}^{2},\delta_{k}>_{Q} -\mbf{w}\bullet \delta_{k}= rank(\mc{I}_{k}(\mbf{v}^{1}))-2g_{k}+2.
  \end{align*}
  Hence we have
  \begin{align*}
    vdim(\mb{P}_{\mf{M}(\mbf{v}^{1},\mbf{w})\times \mbf{L}_{k}}(\mc{I}_{k}^{\vee}(\mbf{v}^{1})))=vdim(\mb{P}_{\mf{M}(\mbf{v}^{2},\mbf{w})\times \mbf{L}_{k}}(\mc{I}_{k}(\mbf{v}^{2})))= \\
    \frac{1}{2}(dim(\mf{M}(\mbf{v}^{2},\mbf{w}))+dim(\mf{M}(\mbf{v}^{1},\mbf{w}))+2g_{k}+2)=dim(\mf{P}(\mbf{v}^{1},\mbf{v}^{2})).
  \end{align*}
  by \cref{020527}. Thus they are isomorphic as smooth schemes. 
\end{proof}
\begin{lemma}
  \label{020543}
  Let $X$ and $Y$ be two smooth schemes, and $Z$ be a smooth Lagrangian of $X\times Y$. Then we have $L_{Z/X}\cong T_{Z/Y}[1]$.
\end{lemma}
\begin{proof}
  It follows from the short exact sequence
  $$0\to TZ \to TY\oplus T^{*}X \to T^{*}Z\to 0.$$
\end{proof}
Finally we choose a dimension vectors  $\mbf{v}_{0}+\delta_{k}=\mbf{v}^{1}=\mbf{v}^{2}-\delta_{k}$. We apply the homological projective duality theorem \cref{0205410} to $\mf{J}_{k}(\mbf{v}^{1})$:
\begin{lemma}
  \label{0206412}
  We have
  \begin{align*}
    \mf{J}_{k}(\mbf{v}^{1})&\cong \mb{P}_{\mf{P}(\mbf{v}^{0},\mbf{v}^{1})}(T_{q\times t(\mbf{v}^{0},\mbf{v}^{1})}[1]) \cong \mb{P}_{\mf{P}(\mbf{v}^{1},\mbf{v}^{2})}(T_{p\times t(\mbf{v}^{1},\mbf{v}^{2})}[1]) \\
    & \cong \mb{P}_{\mf{P}(\mbf{v}^{0},\mbf{v}^{1})}(L_{p(\mbf{v}^{0},\mbf{v}^{1})})\cong  \mb{P}_{\mf{P}(\mbf{v}^{1},\mbf{v}^{2})}(L_{q(\mbf{v}^{1},\mbf{v}^{2})}).
  \end{align*}
  as smooth schemes.
\end{lemma}
\begin{proof}
  By \cref{020527}, $\mf{P}(\mbf{v}^{1},\mbf{v}^{2})$ is a Lagragian of $\mf{M}(\mbf{v}^{1},\mbf{w})\times \mf{M}(\mbf{v}^{2},\mbf{w})\times \mbf{L}_{k}$. By \cref{020543}, we have
  $$T_{q\times t(\mbf{v}^{1},\mbf{v}^{2})}[1]\cong L_{p(\mbf{v}^{1},\mbf{v}^{2})},\quad T_{p\times t(\mbf{v}^{0},\mbf{v}^{1})}[1]\cong L_{q(\mbf{v}^{0},\mbf{v}^{1})}.$$
  By \cref{0205410}, $\mf{J}_{k}(\mbf{v}^{1})$ is the underlying scheme of
  $$ \mb{P}_{\mf{P}(\mbf{v}^{1},\mbf{v}^{2})}(T_{q\times t(\mbf{v}^{1},\mbf{v}^{2})}[1]) \cong \mb{P}_{\mf{P}(\mbf{v}^{0},\mbf{v}^{1})}(T_{p\times t(\mbf{v}^{0},\mbf{v}^{1})}[1]). $$
  Hence to prove the isomorphism, we only need to show that
  $$dim(\mf{J}_{k}(\mbf{v}^{1}))=vdim(\mb{P}_{\mf{P}(\mbf{v}^{1},\mbf{v}^{2})}(T_{q\times t(\mbf{v}^{1},\mbf{v}^{2})}[1])),$$
  which follows from \cref{020531} and  \cref{020527}.
\end{proof}

  \appendix
  \section{Smoothness, Quasi-smoothness, and Virtual Dimension}
 By Lemma 2.1.2 of \cite{Arinkin2015} or \cite{lurie}, A smooth derived scheme is classical, and a smooth classical scheme is smooth as a derived scheme.

By Corollary 2.1.11 of \cite{Arinkin2015}, a morphism $f:Z_{1}\to Z_{2}$ of derived schemes is quasi-smooth if and only if, Zariski-locally on the source, it can be induced in a diagram
  \begin{equation}
    \label{020644}
    \begin{tikzcd}
      Z_{1}\ar{r}{f'} \ar{d} & Z_{2}\times \mb{A}^{n}\ar{r}{pr} \ar{d} & Z_{2} \\
      pt \ar{r}{0} & \mb{A}^{m}.
    \end{tikzcd}
  \end{equation}
  in which the square is Cartesian and $f'$ is a closed embedding. Moreover, if $f$ is a closed embedding, Zariski-locally on the source it can be induced in a diagram \cref{020644} such that $n=0$. As a corollary, a quasi-smooth scheme $X$ is classical if and only if $vdim(X)=dim(\pi_{0}(X))$. In this situation, $X$ is a local complete intersection scheme. A quasi-smooth scheme $X$ is smooth if and only if $vdim(X)=dim(\pi_{0}(X))$ and $\pi_{0}(X)$ is smooth.

The following argument is standard (see Claim 6.5 of \cite{neguct2018hecke}, Exercise 12.2.C of \cite{vakil2017rising} and e.t.c.)
\begin{lemma}
    Consider a Cohen-Macaulay (resp. complete intersection, regular) local ring $R$ and a collection of elements
$f_{1},\cdots, f_{n} \in R$ such that the quotient ring $R/(f_{1},\cdots, f_{n})$ has codimension $n$ in $R$, then for any $1\leq i\leq n$, the quotient $R/(f_{1},\cdots,f_{i})$ is also Cohen-Macaulay (resp. complete intersection, regular) of codimension $i$ in $R$.
\end{lemma}
\begin{corollary}
    \label{gluingscheme}
    Let $f:X\to Y$ be a quasi-smooth closed embeddings of quasi-smooth derived schemes. If $Y-X$ is classical (resp. smooth) and not empty, and $X$ is classical (resp. smooth), then $Y$ is also classical (resp. smooth).
\end{corollary}
\section{Jiang's Derived Projectivization}
\label{sec:Jiang}
In this section, we review Jiang's derived projectivization theory \cite{jiang2022derived}. Let $\psi:\mc{F}\to \mc{G}$ be a morphism of locally free sheaves over a derived scheme $X$. Then over $\mb{P}_{X}(\mc{G})$, we consider the morphism $taut_{\psi}$ as the composition of morphisms:
$$\pi^{*}\mc{F}\otimes \mc{O}_{\mb{P}_{X}(\mc{G})}(-1)\to \pi^{*}\mc{G}\otimes \mc{O}_{\mb{P}_{X}(\mc{G})}(-1)\to \mc{O}.$$
Let $\mb{R}\mc{Z}(taut_{\psi}^{\vee})$ be the derived zero locus of $taut_{\psi}^{\vee}$.
Given a derived scheme $X$, we denote a two-term complex as a morphism of locally free sheaves $\psi:\mc{F}\xrightarrow{\psi}\mc{G}$ and also regard it as an object in $Perf(X)$.

\begin{theorem}[\cite{jiang2022derived}]
  \label{0122111}
  Regarding $\psi$ as an object of $Perf(X)$ and moreover $QCoh(X)$, we have $\mb{P}_{X}(\psi)\cong \mb{R}\mc{Z}(taut_{\psi}^{\vee})$, and $\pi_{0}(\mb{P}_{X}(\psi))\cong P_{\pi_{0}(X)}(\pi_{0}(\psi))$. Moreover, we have $\mc{O}_{\mb{P}_{X}(\psi)}(1)\cong \mc{O}_{\mb{P}_{X}(\mc{G})}(1)|_{\mb{P}_{X}(\psi)}$.
\end{theorem}
In the last part of this section, we assume $X$ to be quasi-smooth and $\psi\in Perf(X)$ has Tor-amplitude $[0,1]$.
\begin{corollary}
The derived scheme $\mb{P}_{X}(\psi)$ is also quasi-smooth, and its virtual dimension is $vdim(X)+rank(\psi)-1$. Moreover, $\mb{P}_{X}(\psi)$ is smooth if and only if
  \begin{enumerate}
  \item $X$ is smooth;
  \item $P_{X}(\psi)$ is smooth and its dimension is  $vdim(X)+rank(\psi)-1$.
  \end{enumerate}
\end{corollary}
Over $\mb{P}_{X}(\psi)$, there is a perfect complex $\bar{\psi}:=\{pr_{\psi}^{*}\psi\to \mc{O}_{\mb{P}_{X}(\psi)}(1)\}[-1]$, which also has Tor-amplitude $[0,1]$.
\begin{theorem}[Theorem 4.27 of \cite{jiang2022derived}, Euler fiber sequences]
  The relative cotangent complex $L_{\mb{P}_{X}(\psi)/X}\cong \mc{O}_{\mb{P}_{X}(\psi)}(-1)\otimes \bar{\psi}$. 
\end{theorem}

\begin{theorem}[Lemma 7.3 of \cite{jiang2022derived}, homological projective duality]
  \label{0205410}
  We have a canonical isomorphism
  $$\mb{P}_{\mb{P}_{X}(\psi)}(T_{\mb{P}_{X}(\psi)/X}[1])\cong \mb{P}_{\mb{P}_{X}(\psi^{\vee})}(T_{\mb{P}_{X}(\psi^{\vee})/X}[1])$$
  such that their underlying scheme are both $P_{\pi_{0}(X)}(\pi_{0}(\psi))\times_{\pi_{0}(X)}P_{\pi_{0}(X)}(\pi_{0}(\psi^{\vee}))$. Moreover, we have
  \begin{align*}
    \mc{O}_{\mb{P}_{\mb{P}_{X}(\psi)}(T_{\mb{P}_{X}(\psi)/X})}(1)\cong \mc{O}_{\mb{P}_{\mb{P}_{X}(\psi^{\vee})}(T_{\mb{P}_{X}(\psi^{\vee})/X})}(1)
  \end{align*}
  and there restriction to $P_{\pi_{0}(X)}(\pi_{0}(\psi))\times_{\pi_{0}(X)}P_{\pi_{0}(X)}(\pi_{0}(\psi^{\vee}))$ are both
  $$\mc{O}_{P_{\pi_{0}(X)}(\pi_{0}(\psi))}(1)\otimes \mc{O}_{P_{\pi_{0}(X)}(\pi_{0}(\psi^{\vee}))}(1).$$
Moreover, both $\mb{P}_{\mb{P}_{X}(\psi)}(T_{\mb{P}_{X}(\psi)/X}[1])$ and $\mb{P}_{\mb{P}_{X}(\psi^{\vee})}(T_{\mb{P}_{X}(\psi^{\vee})/X}[1])$ are quasi-smooth and their virtual dimension is $vdim(X)-1$.
\end{theorem}

\section{The Local Model of Derived Blow-ups}
In this section, we review Hekking's derived blow-up theory, focusing on the case when $\Phi:Y\to X$ is a closed embedding of quasi-smooth schemes. First, we explained the relationship between the classical blow-up and derived blow-up:
\begin{theorem}[Theorem 3.5.5 of \cite{Hekking_2022}]
  Given a closed morphism $\Phi:X\to Y$, $Bl_{\pi_{0}(\Phi)}$ is a schematic closure of $\pi_{0}(X-Y)$ in $\pi_{0}(\mb{B}l_{\Phi})$.
\end{theorem}

\begin{corollary}
    \label{020651}
  If $\pi_{0}(Y)-\pi_{0}(X)$ is not empty and $\mb{B}l_{\Phi}$ is smooth, then $Bl_{\pi_{0}(\Phi)}\cong \mb{B}l_{\Phi}$ and they are both smooth schemes.
\end{corollary}

Now we consider a local model: let $Z$ be a smooth variety, and we consider morphisms of locally free sheaves:
$$g:V\to \mc{O}_{X},\quad h:W\to \mc{O}_{X}, \quad \phi:W\to V$$
such that $g\circ \phi=h$. Let
$$X:=\mb{R}\mc{Z}(g^{\vee}),\quad Y:=\mb{R}\mc{Z}(h^{\vee})$$
be the derived zero locus of $g^{\vee}$ and $h^{\vee}$ respectively. Then $X$ is a closed derived subscheme of $Y$. We denote
$$\Phi:X\to Y,\quad \mf{g}:X\to Z,\quad \mf{h}:Y\to Z$$
as the respective closed embeddings.

Let $\bar{V}$ be the kernel of $taut_{V}:pr_{V}^{*}V\to \mc{O}_{\mb{P}_{Z}(V)}(1)$. Let $v$ be the rank of $V$. We recall Serre's theorem
\begin{theorem}[Serre]
  \label{thm:serre}
  We have
  \begin{equation*}
    \mbf{R}pr_{V*}\mc{O}_{\mb{P}_{Z}(V)}(l)\cong
    \begin{cases}
      Sym^{n}V & l\geq 0 \\
      0 & -v<l<0 \\
       det(V)^{-1}Sym^{-l-v}(V^{\vee})[-v+1]  & l\leq -v
    \end{cases}
  \end{equation*}
\end{theorem}
\begin{corollary}
  Given $0\leq l< v$ and $0\leq k< v$, we have
  \begin{equation}
    \label{e020661}
    \mbf{R}pr_{V}^{*}(\wedge^{k}\bar{V}\otimes \mc{O}_{\mb{P}_{Z}(V)}(-l))=\delta_{kl}\mc{O}_{Z}[-k].
  \end{equation}
\end{corollary}

Let $\bar{g}:\bar{V}\to \mc{O}_{\mb{P}_{Z}(V)}$ be the morphism of $pr_{V}^{*}g$ restricted to $\bar{V}$.

\begin{lemma}[Example 3.6 of \cite{hekking2022stabilizer}]
  The derived blow up $\mb{B}l_{\mf{g}}$ is the derived locus of $\bar{g}^{\vee}$ and
  $$\mc{O}_{\mb{B}l_{\mf{g}}}(-1)\cong\mc{O}_{\mb{P}_{Z}(V)}(-1)|_{\mb{B}l_{\mf{g}}}.$$
\end{lemma}
\begin{lemma}
  \label{020663}
   Given $-v<l\leq 0$, we have
  $$\mbf{R}pr_{\mf{g}*}\mc{O}_{\mb{B}l_{\mf{g}}}(l)\cong \mc{O}_{Z}.$$
\end{lemma}
\begin{proof}
  It follows from the fact the structure sheaf $\mc{O}_{\mb{B}l_{\mf{g}}}$ is resolved by the Koszul complex of $\bar{g}$ and \cref{e020661}.
\end{proof}

On $\mb{P}_{Z}(V)$, we consider the morphism $h':W\otimes \mc{O}_{\mb{P}_{Z}(V)}(-1)\to \mc{O}_{\mb{P}_{Z}(V)}$ where $h'=\mc{O}_{\mb{P}_{Z}(V)}(-1)\otimes (taut_{V}\circ pr^{*}\phi)$. Let $\bar{h}$ be the restriction of $h'$ to $\mb{B}l_{\mf{g}}$. The following theorem is a reformulation of Proposition 3.13 of \cite{hekking2022stabilizer}:
\begin{lemma}
  \label{020664}
  The derived scheme $\mb{B}l_{\Phi}$ is the derived zero locus of $\bar{h}^{\vee}$ and
  $$\mc{O}_{\mb{B}l_{\Phi}}(-1)\cong \mc{O}_{\mb{B}l_{g}}(-1)|_{\mb{B}l_{\phi}}.$$
\end{lemma}

\begin{corollary}
  \label{cor11}
  The derived scheme $\mb{B}l_{\Phi}$ is also quasi-smooth, and the (virtual) exceptional divisor is $\mb{P}_{X}(C_{\Phi})$.
\end{corollary}

Let $w$ be the rank of $W$. Let $pr_{\Phi}$ be the projection from $\mb{B}l_{\Phi}$ to $Y$. We notice that $C_{\Phi}=\phi|_{Y}$, and $vdim(Y)-vdim(X)=rank(C_{\Phi})=v-w$. Let $K_{Y}$ and $K_{\mb{B}l_{\Phi}}$ be the determinant line bundle of $L_{Y}$ and $L_{\mb{B}l_{\Phi}}$.
\begin{theorem}[Discrepancy formula]
   \label{lem:1.5}
   We have
   $$K_{\mb{B}l_{\Phi}}\cong pr_{\Phi}^{*}K_{Y}\otimes \mc{O}_{\mb{B}l_{\Phi}}(vdim(X)-vdim(Y)+1).$$
\end{theorem}
\begin{proof}
  First, we notice that
  $$K_{Y}\cong det(W)^{-1}(det K_{Z})|_{Y}.$$
  Second we notice that $\mb{B}l_{\Phi}$ is the derived zero locus of $(h'\oplus \bar{g})^{\vee}$ over $\mb{P}_{Z}(W)$. Thus we have
  \begin{align*}
    K_{\mb{B}l_{\Phi}}&\cong \mc{O}_{\mb{P}_{Z}(V)}(w)\otimes pr_{V}^{*}det(W)^{-1} \otimes det(V)^{-1}\otimes \mc{O}_{\mb{P}_{Z}(V)}(1)\otimes K_{\mb{P}_{Z}(V)}|_{\mb{B}l_{\Phi}} \\
                       &\cong  pr_{V}^{*}(det(W)^{-1}(det K_{Z})) \mc{O}_{\mb{P}_{Z}(V)}(w-v+1)|_{\mb{B}l_{f}} \\
    & \cong pr_{\Phi}^{*}K_{Y}\otimes \mc{O}_{\mb{B}l_{\Phi}}(vdim(X)-vdim(Y)+1).
  \end{align*}
\end{proof}
\begin{theorem}
  \label{020655}
  Given $-rank(C_{\Phi})<l\leq  0$, we have
  $$\mbf{R}pr_{\Phi*}\mc{O}_{\mb{B}l_{\Phi}}(l)\cong \mc{O}_{Y}.$$
\end{theorem}
\begin{proof}
  It follows from \cref{020663} and the Koszul resolution of $\mc{O}_{\mb{B}l_{\Phi}}$ by \cref{020664}.
\end{proof}
\begin{corollary}[Grauert-Riemenschneider vanishing theorem]
  \label{grau}
  We have $\mbf{R}pr_{\Phi*}K_{\mb{B}l_{\phi}}\cong K_{Y}$.
\end{corollary}

 We have the following commutative diagram:
\begin{equation*}
  \begin{tikzcd}
    \mb{P}_{X}(C_{\Phi})\ar{r}\ar{d} & \mb{B}l_{\Phi}\ar{d} \\
    X \ar{r}{\Phi} & Y
  \end{tikzcd}
\end{equation*}
which induces an isomorphism $Y-X\cong \mb{B}l_{\Phi}- \mb{P}_{X}(C_{\Phi})$.
\begin{theorem}[Gluing theorem]
  \label{020656}
  We have
  \begin{enumerate}
  \item \label{classical} if both $\mb{P}_{X}(C_{\Phi})$ and $Y-X$ are classical and not empty, then both $X$ and the derived blow up $\mb{B}l_{\Phi}$ are also classical, and $\mb{P}_{X}(N_{\phi})$ is a Cartier divisor of $\mb{B}l_{\Phi}$;
  \item \label{assume}if both $\mb{P}_{Y}(C_{\Phi})$ and $Y-X$ are smooth and not empty, then $X$ and the derived blow up $\mb{B}l_{\Phi}$ are also smooth;
  \item assuming (\ref{assume}) and $dim(X)<vdim(Y)$, then $Y$ is also classical and moreover is a canonical singularity. 
  \end{enumerate}
\end{theorem}
\begin{proof}
    Both (\ref{classical}) and (\ref{assume}) follows from \cref{gluingscheme}. Assuming (\ref{assume}), if $dim(X)<vdim(Y)$, we have $dim(\pi_{0}(Y))=dim(Y)$ and thus $Y$ is classical, too. By \cref{020655}, $Y$ is a rational singularity, and thus a canonical singularity.
\end{proof}

Finally, we notice that for any closed morphism of quasi-smooth schemes $\Phi:X\to Y$, by \cref{020644} it is always locally represented by the local models. Thus \cref{cor11}, \cref{020655}, and \cref{020656} always hold, and \cref{lem:1.5} and \cref{grau} hold if $\mb{B}l_{f}$ is smooth.
\begin{example}
    Let $f:X\to Y$ and $g:Y\to Z$ be a closed embedding of classical varieties such that
    \begin{enumerate}
        \item $X$ and $Z$ are smooth and $Y$ is l.c.i;
        \item $dim(Z)=dim(Y)+1=dim(X)+2$;
        \item the closed embedding $g$ is quasi-smooth; 
    \end{enumerate}
    It induces a morphism of locally free sheaves $C_{g}|_{X}\to C_{g\circ f}$ and we denote $T$ as the zero locus of this morphism (we notice that $C_{g}$ has rank $1$). If $T$ is smooth and $dim(T)=dim(X)-2$, then $\mb{B}l_{f}$ is also smooth and induces a crepant resolution of $Y$. 
\end{example}
\section{Derived Blow-up and Virtual Fundamental Class}
\label{sec:app}
Let $X$ be a quasi-smooth scheme. We denote
$$D^{b}(X):=\{\mc{F}\in QCoh(X)|H^{i}(\mc{F})\in Coh(\pi_{0}(X)) \text{ and are bounded}\}$$
as a subcategory of $QCoh(X)$. We notice that $D^{b}(X)$ has a standard $t$-structure such that the heart is $Coh_{\pi_{0}(X)}$. Hence the Grothendieck group of $D^{b}(X)$, which we denote as $K_{0}(D^{b}(X))$ is isomorphic to $G_{0}(\pi_{0}(X))$, i.e. the Grothendieck group of coherent sheaves on $\pi_{0}(X)$. Moreover, as $\mc{O}_{X}\in D^{b}(X)$, it generates a class of $G_{0}(\pi_{0}(X))$, which we denote as $[X]_{K}^{vir}$, following the notation of \cite{ciocan-fontanine07:virtual}.

Let $Z$ be a smooth scheme, with a morphism of locally free sheaves $\mf{h}:V\to \mc{O}_{Z}$. Let $Y$ be the derived zero locus of $\mf{h}^{\vee}$ and $g:X\to Z$ be a smooth closed subscheme of $Z$ such that $\mf{h}|_{X}=0$. Then we have the closed embedding $f:X\to Y$. By restricting to $X$, the co-normal complex $C_{f}$ is a two-term complex
$$\{V\to C_{g}\}.$$
Let $pr_{g}$ be the projection morphism from $Bl_{g}$ to $Z$. Then the image of the morphism
$$pr_{g}^{*}\mf{h}:pr_{g}^{*}V\to \mc{O}_{Bl_{g}}$$
is inside $\mc{O}_{Bl_{g}}(1)$. Hence we denote
$$\bar{h}:pr_{g}^{*}V\otimes \mc{O}_{Bl_{g}}(-1)\to \mc{O}_{Bl_{g}}$$
as the induced morphism.

Like \cref{020664}, the following lemma is another reformulation of Proposition 3.13 of \cite{hekking2022stabilizer}:
\begin{lemma}
  The derived blow up $\mb{B}l_{f}$ is the zero locus of $\bar{h}^{\vee}$, and
  $$\mc{O}_{\mb{B}l_{f}}(-1)\cong \mc{O}_{Bl_{g}}(-1)|_{\mb{B}l_{f}}.$$
\end{lemma}
Let $pr_{f}:\mb{B}l_{f}\to Y$ be the projection. Let $r$ be the rank of $C_{f}$, and $r=dim(X)-vdim(Y)$. We assume $r\leq 0$. 
\begin{theorem}
  \label{app1}
  There exists $D_{a}\in D^{b}(Y)$, $-r-1\leq a\leq  0$ such that
  \begin{enumerate}
  \item $D_{-r-1}\cong\mc{O}_{Y}$ and $D_{0}\cong \mbf{R}pr_{f*}\mc{O}_{\mb{B}l_{f}}$;
  \item We have triangles:
    \begin{equation}
      \label{eq:triangle}
     D_{a-1}\to D_{a}\to \mbf{R}f_{*}(det(C_{f})^{-1}\wedge^{l-r}(C_{f}^{\vee}[1]))[1-a]
    \end{equation}
  \end{enumerate}
\end{theorem}
\begin{proof}
  Let $\bar{Y}:=Y\times_{Z}^{\mb{L}}Bl_{g}$.  Then $Bl_{f}$ is a closed derived scheme of $\bar{Y}$, and let $\bar{g}$ as the projection from $\bar{Y}$ to $Y$.  Given an integer $a\leq 0$, we define the complex of coherent sheaves $\mc{A}^{a}$ such that the cohomological degree $m$ element is 
  \begin{equation}
    (\mc{A}^{a})_{m}=
    \begin{cases}
      0 & m\geq a,\\
      \wedge^{-m}(pr_{g}^{*}V)\otimes \mc{O}_{\mb{P}_{Y}(C_{g})}(m-a) & m<a,
    \end{cases}
  \end{equation}
  where the differential is the wedge powers of $\bar{h}$ restricted to $\mb{P}_{Y}(C_{g})$. Moreover, $\mc{A}^{a}$ is a dg-coherent sheaf over $\bar{Y}$.
  Given an integer $a\leq 0$, we also define the complex of coherent sheaves such that the cohomological degree $m$ element is 
  \begin{equation}
    (\mc{D}^{a})_{m}= 
    \begin{cases}
      0 & m>1, \\
      \wedge^{-m}(pr_{g}^{*}V) & 0\geq m\geq a, \\
      \wedge^{-m}(pr_{g}^{*}V)\otimes \mc{O}_{\mb{B}l_{f}}(m-a) & m<a,
    \end{cases}
  \end{equation}
  where the differential is the wedge power of $\bar{h}$ when $m<a$ and the differential is the wedge power of $pr_{g}^{*}\mf{h}$ when $m\geq a$. Then $\mc{D}^{a}$ is also a dg-coherent sheaf over $\bar{Y}$. We define
  $$D^{a}:=\mbf{R}\bar{g}_{*}\mc{D}^{a}.$$
  By Serre's theorem, we have
  $$D^{0}\cong \mbf{R}pr_{f*}\mc{O}_{\mbf{R}Bl_{f}},\quad D^{-r-1}\cong \mc{O}_{Y}.$$
  We notice that $\mc{A}^{a}\cong \mc{D}^{a}/\mc{D}^{a-1}$. Moreover, by Serre's theorem, we have
  $$\mbf{R}\bar{g}_{*}\mc{A}^{a}\cong \mbf{R}f_{*}(det(C_{f})^{-1}\wedge^{a+r}(C_{f}^{\vee}[1]))[1-a].$$
\end{proof}
\begin{corollary}
  \label{app2}
  Let $[Y]_{K}^{vir}$ and $[\mb{B}l_{f}]_{K}^{vir}$ be the $K$-theoretic virtual fundamental class of $\pi_{0}(Y)$ and $\pi_{0}(\mb{B}l_{f})$ respectively. Then we have
  $$\pi_{0}(pr_{f})_{*}[\mb{B}l_{f}]_{K}^{vir}=[Y]_{K}^{vir}+\pi_{0}(f)_{*}(\sum_{a=-r}^{0}(-1)^{1-a}det(C_{f})^{-1}\wedge^{a+r}(C_{f}^{\vee}[1]))$$
\end{corollary}

\bibliography{Quadruple}
\bibliographystyle{plain}
\end{document}